\documentclass[11pt]{article}

\usepackage{amssymb,amsmath,amsthm,enumitem,tikz,caption}

\numberwithin{equation}{section}

\newtheorem{thm}{Theorem}
\newtheorem{prp}{Proposition}[section]
\newtheorem{lmm}[prp]{Lemma}
\newtheorem{crl}[prp]{Corollary}

\newtheorem{cnj}[prp]{Conjecture}

\theoremstyle{definition}
\newtheorem{dfn}[prp]{Definition}
\newtheorem{eg}[prp]{Example}

\theoremstyle{remark}
\newtheorem{rmk}[prp]{Remark}

\def\BE#1{\begin{equation}\label{#1}}
\def\EE{\end{equation}}
\def\eref#1{(\ref{#1})}

\def\BEnum#1{\begin{enumerate}[label=#1,leftmargin=*,topsep=-10pt,itemsep=-3pt]}
\def\EEnum{\end{enumerate}}

\def\ov#1{\overline{#1}}

\def\wt#1{\widetilde{#1}}
\def\tn#1{\textnormal{#1}}

\def\lra{\longrightarrow}

\def\C{\mathbb C}

\def\fM{\mathfrak M}

\def\cN{\mathcal N}
\def\cO{\mathcal O}

\def\R{\mathbb R}

\def\cT{\mathcal T}

\def\Z{\mathbb Z}

\def\fd{\mathfrak d}
\def\nd{\textnormal{d}}
\def\fd{\mathfrak d}
\def\ff{\mathfrak f}

\def\na{\nabla}

\def\ze{\zeta}

\def\Ga{\Gamma}

\def\Si{\Sigma}

\def\cok{\tn{cok}}
\def\nd{\tn{d}}
\def\ev{\tn{ev}}

\def\ord{\tn{ord}}
\def\supp{\tn{supp}}

\def\prt{\partial}
\def\dbar{\ov\partial}

\def\supp{\textnormal{supp}}

\def\M{\mathfrak{M}}
\def\f{\mathfrak{f}}

\def\d{\textnormal{d}}
\def\dvol{\textnormal{d}vol}

\def\volumeformm{\nu}
\def\bydefinition{:=}
\def\bydefinitiontwo{:=}

\def\curindby#1{T_{#1}}
\def\conestr#1{\Lambda_{#1}}

\newcounter{temp}

\title{Positivity of Intersections and\\Tameness of Almost Complex 4-manifolds}
\author{Spencer Cattalani\thanks{Partially supported by NSF grant DMS 1901979 and by the Simons Foundation}}
\date{\today}

\begin{document}

\maketitle

\begin{abstract}
We prove that pseudoholomorphic curves intersect complex 2-cycles positively in almost complex 4-manifolds. This makes possible a general and conceptually simple proof that an almost complex 4-manifold with many curves admits a taming symplectic structure, as envisioned by Gromov. Furthermore, we prove that the positivity of intersections between pseudoholomorphic curves is stable, in a geometric sense.
\end{abstract}

\tableofcontents

\section{Introduction}
\label{Introduction_sec}
Gromov's introduction of pseudoholomorphic curves to symplectic geometry revolutionized the field. Such techniques are particularly fruitful in dimension 4. Among the many results proved in this setting are the uniqueness of the symplectic structure on $\mathbb{CP}^2$ \cite[Theorem~B]{TaubesUnique}, the classification of rational and ruled symplectic manifolds \cite{McDuff}, and the retractability of the symplectomorphism group of the Fubini-Study form on $\mathbb{CP}^2$ onto its isometry group \cite[Remark~0.3.C]{Gr}. A basic fact underlying this theory is that every symplectic form tames some almost complex structure. This paper is concerned with the converse: when does an almost complex manifold admit a taming symplectic form?
Sullivan provided a framework for answering this question in terms of his notion of \textit{complex cycles}:

\begin{prp}[{\cite[Theorem III.2]{Sullivan}}]\label{taming_prp}
Let $(X,J)$ be a closed almost complex manifold. If \hbox{$[T] \neq 0 \in H_2(X,\R)$} for every nonzero complex cycle $T$ on $X$, then $(X,J)$ is tamed by a symplectic form.
\end{prp}

The definition and basic properties of complex cycles are recalled in Section~\ref{structure_cycles_subsection}. It is also shown in~\cite[Theorem III.2]{Sullivan} that a closed almost complex manifold always contains nonzero complex cycles. They are studied in a similar context in \cite{Cattalani}. Unfortunately, they are generally quite singular; see \cite{Gue} for some pathological examples. One would therefore like to answer the above question in terms of more familiar objects, such as pseudoholomorphic curves.\\

Gromov proposed such an answer in \cite[Remark~2.4.$\textnormal{A}'$]{Gr}. This remark says that an almost complex 4-manifold with many pseudoholomorphic curves should admit a taming symplectic form. The argument presented therein comes from integral geometry. It is carried out in \cite[Section 10]{McKay}. Unfortunately, his reasoning requires the manifold to be swept out by smooth families of transversely intersecting curves. We prove that Gromov's prediction holds in general. All notions appearing in the statement of Theorem \ref{Curves2Symp_thm} below are recalled in Section \ref{deformation_sec}.

\begin{thm}\label{Curves2Symp_thm}
Let $(X,J)$ be a closed almost complex 4-manifold such that every pair of points can be joined by a closed connected pseudoholomorphic curve~$C$. Suppose also that $(X,J)$ contains
\begin{enumerate}[label=(\alph*),leftmargin=*]

\item\label{fc_it} a free pseudoholomorphic curve, or
\item\label{imgn_it} an immersed pseudoholomorphic curve of genus $g$ with $n$ nodes and self-intersection number at least $2g+2n$, or
\item\label{erc_it} an embedded rational curve with nonnegative self-intersection number, or
\item\label{umerc_it} uncountably many embedded rational curves.
\setcounter{temp}{\value{enumi}}
\end{enumerate}
Then $(X,J)$ admits a taming symplectic form $\omega$.
\end{thm}

In case \ref{erc_it} or \ref{umerc_it} holds, it then follows from \cite{McDuff} and \cite[Theorem A]{Wendl} that $\omega$ may be chosen so that $(X,\omega)$ is obtained by blowing up $\mathbb{CP}^2$ with the Fubini-Study form or a symplectic sphere bundle over a surface. We note that the curves $C$ in Theorem \ref{Curves2Symp_thm} are not assumed to be immersed or irreducible. Neither are they assumed to lie in smooth families or to intersect transversely. The sharpness of Theorem \ref{Curves2Symp_thm} is discussed in Section \ref{remarks_sec}.\\

Theorem \ref{Curves2Symp_thm} is a consequence of the simple geometric principle of \textit{positivity of intersections}. Classically, this states that two closed pseudoholomorphic curves which intersect in an almost complex 4-manifold either coincide on an irreducible component or have positive intersection number. We extend this principle to include the case of a pseudoholomorphic curve intersecting a complex cycle in an almost complex 4-manifold. Under the assumptions of Theorem~\ref{Curves2Symp_thm}, we are able to show that every nonzero complex cycle intersects some curve positively and is thus nontrivial in homology. Theorem~\ref{Curves2Symp_thm} then follows from Proposition~\ref{taming_prp}.\\

It is nontrivial to extend even the statement of positivity of intersections to complex cycles, as the ``irreducible components'' of a complex cycle have not yet been defined in a suitable manner. To circumvent this difficulty, we use an alternative characterization in terms of a collection of test forms. Specifically, a 2-form $\varphi$ on an almost complex manifold $(X,J)$ is called \textit{semipositive} if $\varphi(v,Jv) \geq 0$ for any $v \in TX$. Hermitian forms, taming symplectic forms, and the form $i\d z_1 \d \overline{z_1}$ on $\mathbb{C}^n$ are semipositive. They can be constructed locally quite easily, for instance by multiplying a Hermitian form by a nonnegative bump function. This makes them very good probes for the behavior of the tangent planes to holomorphic curves. For example, an irreducible curve $C$ is an irreducible component of a curve~$S$ if and only if there is no semipositive form $\varphi$ such that $\int_C \varphi > 0$ and $\int_{S} \varphi = 0$. The assumption in the following theorem is directly inspired by this fact.

\begin{thm}\label{PosInter_thm}
Let $(X,J)$ be a closed almost complex 4-manifold. Let $C$ be a closed irreducible pseudoholomorphic curve in $X$ and $T$ be a complex cycle on $X$. If $C \cap \supp(T) \neq \emptyset$ and there is a semipositive form $\varphi$ such that
$$\int_C \varphi > 0 \quad \textnormal{and} \quad T(\varphi) = 0,$$
then
$\langle [T], [C] \rangle > 0.$
\end{thm}

If $C \not\subset \supp(T)$, there is a nonnegative bump function $f$ which is positive somewhere on $C$ and is identically zero on $\supp(T)$. If $\omega$ is a Hermitian form on~$X$, then $f\omega$ is a semipositive form satisfying the assumption in Theorem \ref{PosInter_thm}. Therefore, the following corollary holds.

\begin{crl}\label{pos_int_crl}
Let $(X,J)$ be a closed almost complex 4-manifold. Let $C$ be a closed irreducible pseudoholomorphic curve in $X$ and $T$ be a complex cycle on $X$. If $C \,\cap\, \supp(T) \neq \emptyset$ and \hbox{$C \not\subset \supp(T),$} then
$\langle [T], [C] \rangle > 0.$
\end{crl}

The difficulty in proving the positivity of intersections is that the curves can intersect at singular points. To resolve this issue, one deforms the curves locally to get a transverse intersection. The issues with singularities are magnified in the context of Theorem \ref{PosInter_thm}, as a complex cycle can be much more singular than any holomorphic curve. Therefore, we develop a refined deformation scheme, which is essentially founded on the \textit{Oka principle} (alternatively, the \textit{h-principle}); see the fourth and fifth paragraphs of Section~\ref{future_subsection}. Our method also allows one to prove that the positivity of intersections is stable:

\begin{thm}\label{stable_pos_int_thm}
Let $(X,J)$ be a closed almost complex 4-manifold, equipped with a metric. Let $C$ be a closed irreducible pseudoholomorphic curve in $X$. For every $\varepsilon > 0$, there exists $\delta > 0$ such that if $S \subset X$ is a closed pseudoholomorphic curve,
$\langle[S], [C] \rangle \leq 0$,
and at least one point of $S$ is $\delta$-close to $C$, then every point of $C$ is $\varepsilon$-close to $S$.
\end{thm}

Only the equality case in Theorem~\ref{stable_pos_int_thm} is new, because of the classical positivity of intersections. Broadly, Theorem~\ref{stable_pos_int_thm} says that curves which almost intersect must almost coincide on an irreducible component. We thank the anonymous reviewer for stating the next corollary and the explanatory remark which follows it.

\begin{crl}\label{Hausdorff_crl}
Let $(X,J)$ be a closed almost complex 4-manifold, equipped with a metric. Let $C,S \subset X$ be closed pseudoholomorphic curves so that $C$ is irreducible and
$\langle[S], [C] \rangle > 0.$
If $C \not\subset S$, then there exists $\delta > 0$ such that every closed pseudoholomorphic curve $S'$ which is $\delta$-close to $S$ in the Hausdorff metric satisfies 
$\langle[S'], [C] \rangle > 0.$
\end{crl}

Corollary~\ref{Hausdorff_crl} is more difficult to prove than it might first appear. One is tempted to take a limit, but a Hausdorff limit does not generally preserve homology class, and a limit as currents without a bound on area will generally result in a (singular) complex cycle, not a curve.\\

We now briefly remark on the proof. If a pseudoholomorphic curve varies in a large family, then integration along the fiber allows one to define a well-behaved Poincaré dual 2-form, which would imply Theorems \ref{Curves2Symp_thm}, \ref{PosInter_thm}, and \ref{stable_pos_int_thm} via Proposition~\ref{taming_prp}. However, a pseudoholomorphic curve might not vary in a family of pseudoholomorphic curves. We resolve this issue by allowing the curves in the family to fail to be pseudoholomorphic on a small prescribed disk. This allows every pseudoholomorphic curve to vary in an arbitrarily large family, while retaining a great deal of control over the resulting Poincaré dual form. This argument is the key step in the proof of Proposition~\ref{PD_prp}, which quickly implies all of the main results of this paper and should be of independent interest. In fact, one can prove Theorem~\ref{Curves2Symp_thm} directly from Proposition~\ref{PD_prp} by adding together a suitable collection of Poincaré dual forms to yield a symplectic form; this method avoids the use of complex cycles.\\

The outline of the paper is as follows. In Section \ref{remarks_sec}, we discuss possible extensions of the main results of this paper, as well as their connection to other work. In Section \ref{prelim_sec}, we recall some basics from the theory of currents. In Section \ref{deformation_sec}, we recall the deformation theory of pseudoholomorphic curves and state Proposition \ref{CurveMS_prp}, which describes the flexibility of pseudoholomorphic curves with boundary. It is proved in the appendix. In Section \ref{PD_sec}, we prove Proposition \ref{PD_prp}, which constructs a Poincaré dual to a pseudoholomorphic curve. Theorems~\ref{Curves2Symp_thm}, \ref{PosInter_thm}, and \ref{stable_pos_int_thm} follow quickly from Proposition~\ref{PD_prp} and are proved in Section \ref{main_proofs_sec}.\\

\textbf{Acknowledgements.} The author is thankful to Eric Bedford, Dennis Sullivan, Dror Varolin, and Scott Wilson for helpful conversations, to Ethan Addison for his comments on an earlier draft, and to Aleksey Zinger for his great help with the exposition, especially in the appendix. The author is also grateful to the anonymous reviewer for a thorough report that significantly improved the clarity of writing, particularly of the introduction and Section~\ref{PD_sec}.

\section{Discussion of results}\label{remarks_sec}

In this section, we explain some features of the main results and their proofs, pose Conjecture~\ref{cone_cnj} generalizing Theorem~\ref{Curves2Symp_thm}, and indicate connections to related work.
\subsection{Sharpness}
Theorem~\ref{Curves2Symp_thm} is not sharp, as there are many tamed almost complex manifolds which admit no pseudoholomorphic curves at all (e.g. a general complex torus). However, Example~\ref{hopf_eg} and Propositions~\ref{untameble_prp} and \ref{lefschetz_prp} below show that the assumptions cannot be significantly weakened. First of all, Theorem~\ref{Curves2Symp_thm} would not hold if $X$ were assumed only to have a pseudoholomorphic curve passing through each point, as demonstrated by the following example.

\begin{eg}\label{hopf_eg}
Let $S \bydefinition (\C^2 - \{0\}) / \{z \sim 2z\}$ be the Hopf surface. It has a holomorphic curve through every point and satisfies \ref{fc_it} in Theorem \ref{Curves2Symp_thm}, but $H^2(S,\Z) = 0$, so $S$ is not tamed.
\end{eg}

One might hope that there is a topological condition on $X$ or a geometric condition concerning just one pseudoholomorphic curve (e.g. that it be sufficiently free) that would guarantee that $(X,J)$ is tamed. By Proposition \ref{untameble_prp} below, there is no such condition.

\begin{prp}\label{untameble_prp}
Let $(X,J)$ be an almost complex manifold of dimension at least 4. For every nonempty open $U \subset X$, there exists an almost complex structure $J'$ on $X$ which is deformation equivalent to $J$, is equal to $J$ outside of $U$, and admits no taming symplectic structure.
\end{prp}

\begin{proof}
First, we construct an almost complex structure $J'$ on $\mathbb{C}^n$ for $n \geq 2$ so that $J'$ equals the standard complex structure outside the unit ball and $\mathbb{C}^n$ contains a $J'$-holomorphic embedded torus.
Let $\C^n$ have complex coordinates $x_1 + iy_1, \dots, x_n + iy_n$ and $\R^3$ have real coordinates $x_1, y_1$, and $x_2$. Embed the 2-torus $T^2$ in the unit ball in $\R^3$. By the Gauss-Bonnet Theorem, the Gauss map $\nu: T^2 \lra S^2$, defined by taking the unit normal vector, has degree zero (see \cite[Lemma~6.3]{Milnor}) and is thus null-homotopic. Therefore, one can extend the map $\nu$ to the rest of the unit ball in $\C^n$, so that it remains in the span of $\frac{\partial}{\partial x_1}$, $\frac{\partial}{\partial y_1}$, and $\frac{\partial}{\partial x_2}$ and restricts to the constant vector field $\frac{\partial}{\partial x_2}$ on the unit sphere. Extend it to the rest of $\C^n$ as a constant vector field. Let $v$ be the resulting vector field and define a field of oriented planes $\pi$ on $\C^n$ by $v \wedge \frac{\partial}{\partial y_2}$. Define an almost complex structure~$J'$ on $\C^n$ via a rotation by a positive quarter turn in $\pi$, its orthogonal complement in the span of $\frac{\partial}{\partial x_1}$, $\frac{\partial}{\partial y_1}$, $\frac{\partial}{\partial x_2}$, and $\frac{\partial}{\partial y_2}$, and in each plane spanned by $\frac{\partial}{\partial x_k}$ and $\frac{\partial}{\partial y_k}$ for $2 < k \leq n$. This $J'$ equals the standard complex structure outside the unit ball and contains a $J'$-holomorphic embedded torus.\\

Now, we modify $(X,J)$ using this local model. Deform $J$ on a nonempty open subset $U' \subset U$ so that the resulting structure is integrable on $U'$. Replace $J$ with the $J'$ constructed above inside a complex coordinate chart. As $J'$ is standard outside of the unit ball, this yields an almost complex structure~$J'$ on $X$. The manifold $X$ contains a null-homologous nonconstant $J'$-holomorphic curve, so $J'$ is not tamed by any symplectic structure.
\end{proof}

As noted above, there are almost complex 4-manifolds which satisfy the conclusion, but not the assumptions of Theorem~\ref{Curves2Symp_thm}. However, Proposition~\ref{lefschetz_prp} below shows that every almost complex 4-manifold satisfying the conclusion of Theorem \ref{Curves2Symp_thm} can be transformed into one which satisfies the assumptions of Theorem \ref{Curves2Symp_thm}. In particular, it shows that Theorem \ref{Curves2Symp_thm} is sharp up to blowups and deformations.

\begin{prp}\label{lefschetz_prp}
Let $(X,J)$ be a closed tamed almost complex 4-manifold. There are a taming symplectic form $\omega$ on $X$, a symplectic manifold $(X',\omega')$ obtained by symplectically blowing up $(X,\omega)$ at a finite number of points, and an almost complex structure $J'$ on $X'$ tamed by $\omega'$ such that every pair of points of $X'$ can be joined by a closed connected $J'$-holomorphic curve~$C$ and $X'$ contains a free $J'$-holomorphic curve.
\end{prp}

\begin{proof}
Taming $J$ is an open and positive scaling invariant condition on symplectic forms. Therefore, if $(X,J)$ is tamed, it admits a taming symplectic form $\omega$ such that $[\omega] \in H^2(X,\R)$ is the reduction of an integer cohomology class. Therefore, by \cite[Theorem 2]{DonLefschetz}, $(X,\omega)$ admits a Lefschetz pencil such that the fibers are Poincaré dual to a multiple of $[\omega]$. Since $\omega$ has positive self-intersection, the base locus is nonempty. It follows from the definition of Lefschetz pencil that the base locus is finite and the fibers intersect transversely. Symplectically blowing up along the base locus of the pencil yields a symplectic manifold $(X',\omega')$ which admits a Lefschetz fibration; see \cite[Chapter~3]{Wendl} for a full exposition of this construction.\\

An exceptional divisor of this blowup is an embedded symplectic sphere with self-intersection $-1$. There is an almost complex structure on $X'$ tamed by $\omega'$ in which the exceptional divisor is a pseudoholomorphic curve. As it is a rational curve with self-intersection $-1$, it is regular by \cite[Theorem~1]{HLS}. Therefore, by Gromov's Compactness Theorem \cite[Theorem~1.5.B]{Gr}, every almost complex structure on $X'$ which is tamed by $\omega'$ has a pseudoholomorphic curve in the class of the exceptional divisor. There is an almost complex structure $J'$ tamed by $\omega'$ on $X'$ such that the fibers of the Lefschetz fibration are pseudoholomorphic curves; see the remark following Lemma 2.12 in \cite{DonSmith}. The smooth fibers are free. Furthermore, there is an exceptional divisor intersecting each fiber, so every pair of points of $X'$ can be joined by a closed connected $J'$-holomorphic curve.
\end{proof}

\subsection{Future directions}\label{future_subsection}

Our proof of Theorem \ref{Curves2Symp_thm} via the positivity of intersections is quite robust and many related statements can be proved in a similar manner. The main difficulty is in obtaining a closed semipositive form with which one can use Theorem \ref{PosInter_thm}. In our case, we use free curves to produce such forms. However, they can arise in other ways, for example via pullbacks and limits of symplectic forms. Proposition~\ref{semipositive_cone_prp} below is a somewhat sharper version of Theorem~\ref{Curves2Symp_thm}, albeit with a more technical statement.

\begin{prp}\label{semipositive_cone_prp}
Let $(X,J)$ be a closed almost complex 4-manifold. Suppose there exists a closed semipositive 2-form $\varphi$ such that, for every $x \in X$, there exist a closed connected pseudoholomorphic curve $C_x$ passing through $x$ such that $\int_{C_x} \varphi > 0$. Then, $(X,J)$ admits a taming symplectic structure.
\end{prp}

As in Proposition~\ref{lefschetz_prp}, the existence of a Lefschetz pencil shows that Proposition~\ref{semipositive_cone_prp} is sharp up to deformation; blowing up is not required in this case. It is quite difficult (perhaps impossible) to construct an almost complex 4-manifold with many curves which does not also have closed semipositive forms. Ideally, one could avoid the assumption of semipositivity altogether and use only a \textit{homological} positivity condition, as below. This means that the curves $C_x$ lie in an open cone in $H_2(X,\R)$.

\begin{cnj}\label{cone_cnj}
The word ``semipositive'' can be removed in Proposition \ref{semipositive_cone_prp}.
\end{cnj}

Our method of using complex cycles allows one to leverage other geometric information to construct symplectic forms. For example, as complex cycles are normal currents, they are supported on a set of Hausdorff dimension at least 2; see \cite[Theorem 8.5]{FF}. It follows that sets of lower dimension can be ignored in Theorem \ref{Curves2Symp_thm}. It is not clear how one would see this other than through the perspective of complex cycles. Their geometry is also studied in \cite{Cattalani}. Ultimately, a full understanding of complex cycles would elucidate the role of symplectic structures in almost complex geometry.\\

In the classical proof of the positivity of intersections, the main step is deforming a pseudoholomorphic curve on a small neighborhood. Our proof of Theorem~\ref{PosInter_thm} refines this by deforming the curve everywhere \textit{except} on a small neighborhood. The fact that pseudoholomorphic curves can be deformed this way is a kind of semi-global flexibility. We expect that this point, which can be considered as a type of h-principle or Oka principle, should be of use in a variety of contexts.\\

We now explain a few different ways in which this semi-global flexibility can be understood. The deformations of a pseudoholomorphic curve are controlled by its deformation and obstruction bundles; see \cite[Section~2]{IvSh1}. In this paper, we arrange that the former is large and the latter is zero by choosing appropriate boundary conditions (Lemma~\ref{boundary_conditions_lmm}) and using an automatic transversality result (Proposition~\ref{CurveMS_prp}). This argument is restricted to dimension 4, due to the use of automatic transversality. Alternatively, we could add a family of perturbations to the Cauchy-Riemann equation that covers the obstruction bundle. Essentially by unique continuation, such a family can be chosen to be supported in an arbitrary open set; see \cite[Lemma~4.1]{Zinger}, for example. This technique works in any dimension to produce a family of surfaces which are pseudoholomorphic away from a small open set. There is also a complex-geometric interpretation. An open Riemann surface is Stein \cite[Theorem 8.1.1]{Varolin}. It should then follow from Cartan's Theorem~B \cite[Théorème B]{Serre} that the obstruction bundle vanishes. Therefore, a holomorphic curve with boundary in a complex manifold should deform in an arbitrarily large family.\\

We now briefly remark on connections to other work. The method for constructing compatible symplectic structures in \cite{Taubes} is quite similar to the integral-geometric method proposed in \cite[Remark~2.4.$\textnormal{A}'$]{Gr}. By \cite{TaubesCurrents}, this method can be used to show that almost complex 4-manifolds with sufficiently well-distributed curves admit \textit{compatible} symplectic structures, a stronger condition than just taming. By deforming curves, we avoid having to assume that they are well-distributed (formally, that they can be used to form a \textit{Taubes current}), but we do not produce a compatible symplectic structure. We hope that our method, enriched with analytic estimates as in \cite{Taubes}, could help produce compatible forms, and therefore make progress towards Donaldson's ``tamed-to-compatible'' conjecture \cite{DonConj}. With respect to Theorem \ref{PosInter_thm}, one might wish to understand intersections between two complex cycles. This is akin to the intersection theory of positive $(1,1)$-currents in complex geometry, which is well-developed; see \cite{Bedford, Demailly}, for example. Pluripotential theory is central to this development. Unfortunately, it seems that complex cycles do not always admit potential functions, which significantly hampers the use of pluripotential theory in our context. 

\section{Preliminaries: currents}\label{prelim_sec}

In this section, we recall the necessary background on currents. This plays an essential role in the proof of Proposition~\ref{PD_prp} in Section~\ref{PD_sec}, which constructs a differential form with controlled geometry. We do this in two parts - on a large ``good'' set and a small ``bad'' set. The construction of the form in the good set uses integration along the fiber, which is the focus of Section~\ref{integration_subsection}. The construction on the bad set uses Sullivan's theory of structure currents, which is the focus of Section~\ref{structure_cycles_subsection}.\\

Let $X$ be a smooth $n$-manifold. For $k \in \mathbb{Z}^{\geq 0}$, we denote by $\Omega_c^k(X)$ the space of compactly supported differential $k$-forms on $X$. A \textit{k-current} on an $n$-manifold $X$ is a linear functional on $\Omega^{k}_{c}(X)$ which is continuous in the strong $C^\infty$-topology on $\Omega^{k}_{c}(X)$. A sequence of forms converges in the strong $C^\infty$-topology if it converges in $C^\infty$ and all the forms are supported in the same compact set. We let $(\Omega^{k}_{c}(X))^*$ denote the space of $k$-currents on $X$. The exterior derivative on forms induces a boundary operator,
$$\partial:(\Omega^{k}_{c}(X))^* \lra (\Omega^{k-1}_{c}(X))^*, \quad \partial T(\alpha) \bydefinition T(\d\alpha).$$
A current $T$ on $X$ is called a \textit{cycle} if $\partial T=0$. The \textit{support} of a $k$-current $T$ on $X$, denoted $\supp(T)$, is the complement of the largest open subset $U \subset X$ such that $T(\alpha) = 0$ for all $\alpha \in \Omega^{k}_{c}(U)$.
If $\alpha \in \Omega^{k}(X)$ with $\supp(\alpha) \cap \supp(T)$ compact, then
$$T(\alpha) \bydefinition T(\rho \alpha) \in \R$$
does not depend on the choice of $\rho \!\in\! C^\infty_c(X)$ such that $\rho$ is $1$ on a neighborhood of $\supp(\alpha) \cap \supp(T)$. If $T$ is a $k$-current on $X$ and $U \subset X$ is an open subset, then the restriction of $T$ to $\Omega^{k}_{c}(U)\subset\Omega^{k}_{c}(X)$ is a $k$-current on $U$, which we denote by $T|_U$.
\begin{eg}\label{submanifold_current_eg}
A dimension $k$ closed oriented submanifold $Y$ of $X$ defines a $k$-current 
$$\curindby{Y} :  \Omega_c^{k}(X) \lra \R, \quad \curindby{Y}(\alpha) := \int_Y \alpha.$$
In this case, $\supp(\curindby{Y}) = Y$.
\end{eg}

\begin{eg}
If $X$ is oriented, a degree $(n-k)$ differential form $\omega$ defines a $k$-current
$$\curindby{\omega} :  \Omega_c^{k}(X) \lra \R, \quad \curindby{\omega}(\alpha) := \int_X \omega\wedge\alpha.$$
In this case, $\supp(\curindby{\omega}) = \supp(\omega)$. If $X$ is closed, the inclusion
$$\iota: \Omega^k(X) \lra (\Omega^{n-k}_c(X))^*$$
induces an isomorphism from de Rham cohomology to the homology of currents; see~\cite[Chapter~IV, Theorem~14]{deRham}. The analogous statement in relative homology also holds.
\end{eg}

\begin{eg}
A $k$-multivector $v\!\in\bigwedge^{\!k}TX$ defines a $k$-current
$$\curindby{v}:  \Omega_c^{k}(X) \lra \R, \quad \curindby{v}(\alpha) := \alpha(v).$$
The current $\curindby{v}$ is called a \textit{Dirac current}.
\end{eg}

A current $T$ on $X$ is said to be \textit{smooth} on an open set $U \subset X$ if $T|_U = \curindby{\omega}$ for some differential form $\omega$ on $U$. We note that the current defined by a smooth submanifold of positive codimension, as in Example~\ref{submanifold_current_eg}, is not smooth on any open subset intersecting $Y$.\\

The pullback operation on forms induces a pushforward on currents with compact support. For a smooth map $f: M \lra X$ and a compactly supported $k$-current $T$ on $M$,
\begin{equation}\label{pushforward_eqn}
f_*T: \Omega_c^{k}(X) \lra \R, \quad \{f_*T\}(\alpha) := T(f^*\alpha), 
\end{equation}
is a well-defined $k$-current on $X$.

\subsection{Integration along the fiber}\label{integration_subsection}
In special cases, the pushforward of a smooth current is smooth. Such cases include \textit{integration along the fiber}; see \cite[Section 6]{BottTu} and the proof of Lemma \ref{submersion_pushforward_lmm} below. 

\begin{lmm}\label{submersion_pushforward_lmm}
Let $f\!: M\! \lra \!X$ be a smooth map from an oriented $m$-manifold $M$ with boundary to an oriented $n$-manifold $X$, $k \!\in\! \mathbb{Z}^{\geq 0}$, and $\omega \!\in\! \Omega_c^{m-k}(M)$. If $x \!\in\! X\!-\!f(\partial M)$ is a regular value of $f$, $v \!\in\! \bigwedge^{\!n-k}(T_xX)$, $\widetilde{v} \!\in \!\Gamma(f^{-1}(x); \bigwedge^{\!n-k}TM)$ is a lift of $v$ along $\d f$, and $\omega(-,\wt{v}) \!\in\! \Gamma(f^{-1}(x),\bigwedge^{\!m-n}T^*M)$ is the contraction of $\omega$ with respect to $\wt{v}$ in the last $n\!-\!k$ inputs, then
$$\{f_*\omega\}_x(v) := \int_{f^{-1}(x)}\omega(-,\wt{v})$$
is independent of the choice of the lift of $v$.
If $W \!\subset\! X\!-\!f(\partial M)$ is an open subset consisting of regular values of $f$, then the map
$$\{f_*\omega\}|_W: W \lra \bigwedge\nolimits^{\!n-k}(T^*W), \quad x \mapsto \{f_*\omega\}_x \,,$$
defines a (smooth) differential form on $W$ and
\begin{equation}\label{pushforward_acting_on_vector_eqn}
\{f_*\curindby{\omega}\}|_W = \curindby{\{f_*\omega\}|_W}.
\end{equation}
In particular, the $k$-current $f_*\curindby{\omega}$ is smooth on $W$.
\end{lmm}

\begin{proof}
We assume $m \geq n$ and show first that this result holds for the projection onto the last $n$ coordinates, \hbox{$\pi: \R^m \lra \R^n$}. Denote the first $m-n$ coordinates of $\R^m$ by $p_1, \dots, p_{m-n}$ and the last $n$ coordinates by $x_1, \dots, x_n$. It suffices to prove the result for simple forms, i.e. those of the form $f\d p_\alpha\! \wedge\!\d x_\beta$, where $\alpha$ and $\beta$ are multi-indices. Let $g\d x_\gamma$ be a form on $\R^n$, where $\gamma$ is a multi-index and $g \!\in\!C^\infty_c(\R^n)$. By Fubini's theorem,
\begin{equation*}
\begin{split}
\{\pi_*(\curindby{f\d p_\alpha \wedge \d x_\beta})\}(g\d x_\gamma) &\bydefinition \int_{\R^m} f\d p_\alpha \wedge \d x_\beta \wedge\pi^*(g\d x_\gamma) = \int_{\R^m} f\d p_\alpha \wedge\d x_\beta \wedge (g \circ \pi)\d x_\gamma\\
&= \int_{\R^n} \bigg(\int_{\R^{m-n}} f\d p_\alpha\bigg) \d x_\beta\wedge  (g \d x_\gamma).
\end{split}
\end{equation*}
Thus, $\pi_*(\curindby{f\d p_\alpha \d x_\beta})$ is equal to the current induced by the smooth form
\begin{displaymath}
\bigg(\int_{\R^{m-n}} f\d p_\alpha\bigg) \d x_\beta.
\end{displaymath}
Expressed in this way, it is also clear that it acts on multivectors as described in (\ref{pushforward_acting_on_vector_eqn}).\\

Now, we show the result in the general case. For $x \!\in\! W \!\subset \!X\! - \!f(\partial M)$, let $\rho$ be a nonnegative function on~$X$ supported within $W$ and equal to $1$ on a neighborhood $W'$ of $x$. Let $\omega' \bydefinition (\rho \circ f)\omega$. This form is compactly supported and
\begin{equation}\label{new_form_pushforward_eqns}
\{f_*\omega' \}|_{W'} = \{f_*\omega \}|_{W'}\,, \quad \{f_*\curindby{\omega'} \}|_{W'} = \{f_*\curindby{\omega} \}|_{W'}\,. 
\end{equation}
As $f$ is a submersion on $f^{-1}(W)$, for each $p \in f^{-1}(W) \subset M$, there are open neighborhoods $V_p \subset M$ of $p$ and $W_p \subset W$ of $f(p)$ and diffeomorphisms
$$\varphi: \R^m \lra V_p, \quad \psi: W_p \lra \R^n \quad \textnormal{such that} \quad \psi \circ f \circ \varphi = \pi: \R^m \lra \R^n$$ is the projection onto the last $n$ coordinates. By compactness, finitely many $V_p$ cover $\supp(\omega')$. Using a partition of unity, decompose $\omega'$ as a sum of forms $\omega'_p$ each supported in one of these $V_p$.
Therefore,
$$\{f_*\omega' \}|_{W'} = \sum_p \{f_*\omega_p' \}|_{W'},~~ \{f_*\curindby{\omega'}\}|_{W'} = \sum_p \{f_*\curindby{\omega_p'} \}|_{W'},~~ \textnormal{and}~~ \{f_*\omega_p' \}|_{W'} = \curindby{\{f_*\omega_p' \}}|_{W'}~\forall p;$$
the last statement follows from the previous paragraph. Along with (\ref{new_form_pushforward_eqns}), this implies the claimed equality~\hbox{$\{f_*\curindby{\omega}\}|_W = \curindby{\{f_*\omega\}|_W}$.}
\end{proof}

Integration along the fiber is our main method of constructing differential forms. We will apply it in the following general situation, which covers Example~\ref{group_eg}, Lemma~\ref{free_curve_PD_lmm}, and Proposition~\ref{PD_prp}. We have a manifold $X$, a moduli space $M_0$ of unmarked parameterized submanifolds in $X$ (in our setting, pseudoholomorphic curves), a moduli space $M_1$ of marked submanifolds in $X$, and the maps
$$\mathfrak{f}: M_1 \lra M_0 \quad \textnormal{and} \quad \ev: M_1 \lra X$$
which forget the marked point and evaluate at the marked point, respectively. The map $\ev|_{\f^{-1}(p)}$ is then the parameterization of the submanifold corresponding to $p \in M_0$. If $\dvol$ is a volume form on $M_0$, then we think of the current $\ev_*T_{\f^*\dvol}$ on $X$ as being given by the integral over $\dvol$ of the currents of integration of the submanifolds parameterized by $M_0$. The following lemma justifies this intuition. 

\begin{lmm}\label{integrated_family_lmm}
Let $X$ and $M_0$ be oriented manifolds and $\dvol$ be a compactly supported form of top degree on $M_0$. Let $M_1$ be an oriented manifold with boundary and
$$\f: M_1 \lra M_0 \quad \textnormal{and} \quad \ev: M \lra X$$
be a proper submersion and a smooth map, respectively. For every form $\alpha$ of degree $\dim M_1 -\dim M_0$ compactly supported in $X - \ev(\partial M_1)$,
\begin{equation}\label{integrated_family_eqn}
\{\ev_*T_{\f^*\dvol}\}(\alpha)
= (-1)^{(\dim M_0)(\dim M_1 - \dim M_0)}\int_{p \in M_0}\bigg(\int_{\f^{-1}(p)}\ev^*\alpha \bigg) \dvol.
\end{equation}
\end{lmm}

\begin{proof}
By definition,
$$\{\ev_*T_{\f^*\dvol}\}(\alpha)
= \int_{M_1} \f^*\dvol \wedge \ev^*\alpha = (-1)^{(\dim M_0)(\dim M_1 - \dim M_0)}\int_{M_1} \ev^*\alpha \wedge \f^*\dvol.$$
Integration along the fibers of $\f$ (or Fubini's theorem) yields
$$\int_{M_1}  \ev^*\alpha\wedge \f^*\dvol = \int_{p \in M_0}\bigg(\int_{\f^{-1}(p)}\ev^*\alpha \bigg) \dvol.$$
This equality does not follow directly from (\ref{pushforward_acting_on_vector_eqn}) because $\ev^*\alpha$ may not be compactly supported, so the pushforward of $\curindby{\ev^*\alpha}$ is not a priori well-defined. However, the form $\ev^*\alpha \wedge \f^*\dvol$ is compactly supported in the interior of $M_1$ and the proof of Lemma~\ref{submersion_pushforward_lmm} applies to this case without change.
\end{proof}

Equation (\ref{integrated_family_eqn}) gives geometric information about $\ev_*T_{f^*\dvol}$ as a current. By Lemma~\ref{submersion_pushforward_lmm}, $\ev_*T_{f^*\dvol}$ is smooth on the interior of $X\!-\!\ev(\partial M_1)$ if $\ev$ is a submersion. We would like to use Equation~(\ref{integrated_family_eqn}) to understand the associated differential form. This is the purpose of the following two lemmas, both of which follow from basic linear algebra.

\begin{lmm}\label{invariant_dual_lmm}
Let $(X,J)$ be an almost complex $4$-manifold, $\omega$ be a 2-form on $X$, and $T_\omega$ be the associated current. If $\curindby{\omega}(\alpha) = 0$ for every compactly supported $J$-anti-invariant 2-form $\alpha$, then $\omega$ is $J$-invariant.
\end{lmm}

\begin{proof}
If $\alpha$ is a $J$-anti-invariant 2-form and $\beta$ is a $J$-invariant 2-form, then $\alpha \wedge \beta = 0$. For a 2-form~$\omega$, define
$$\omega_+ = \frac{\omega(-,-) + \omega(J-,J-)}{2} \quad \textnormal{and} \quad \omega_- = \frac{\omega(-,-) - \omega(J-,J-)}{2}.$$
It is clear that $\omega_+$ is $J$-invariant, $\omega_-$ is $J$-anti-invariant, and $\omega = \omega_+ + \omega_-$. If $\omega$ is not $J$-invariant, then $\omega_- \neq 0$. Therefore, there exists a 2-form $\alpha$ such that $\omega_- \wedge \alpha \neq 0$. By the above,
$$\omega \wedge \alpha_- = \omega_- \wedge \alpha_- = \omega_- \wedge \alpha \neq 0.$$
Therefore, after multiplying by a suitable bump function $\rho$, $\rho \alpha_-$ is a compactly supported $J$-anti-invariant 2-form such that $\curindby{\omega}(\rho \alpha_-) \neq 0$.
\end{proof}

\begin{lmm}\label{semipositive_dual_lmm}
Let $(X,J)$ be an almost complex $4$-manifold with positive volume form $\dvol$, $\omega$ be a 2-form on $X$, and $T_\omega$ be the associated current.
If
$$\curindby{\omega}\big(\dvol(\xi\wedge J\xi,-)\big) \geq 0 \quad\forall \xi \in \Gamma_c(X,TX),$$
then $\omega$ is semipositive. In particular, if $T_\omega(\alpha) \geq 0$ for every semipositive $\alpha$, then $\omega$ is semipositive.
\end{lmm}

\begin{proof}
We prove the contrapositive. Suppose there exists a vector $v$ such that $\omega(v,Jv) < 0$. Extend it to a compactly supported vector field $\xi$ on $X$ such that $\omega(\xi(x), J\xi(x)) \leq 0$ for every $x \in X$. Thus,
$$\curindby{\omega}\big(\dvol(\xi(x) \wedge J\xi(x),-)\big)=\int_X \omega \wedge \dvol(\xi(x)\wedge J\xi(x),-) = \int_X\omega(\xi(x),J\xi(x))\dvol < 0,$$
as claimed.
\end{proof}

Below, we prove a special case of Theorem \ref{Curves2Symp_thm} to demonstrate the use of Lemmas \ref{submersion_pushforward_lmm}-\ref{semipositive_dual_lmm}.

\begin{eg}\label{group_eg}
Let $(X,J)$ be a closed almost complex 4-manifold and $G$ be a compact connected Lie group which acts transitively on~$X$ preserving $J$. Let $u: \Sigma \lra X$ be a pseudoholomorphic map from a closed connected Riemann surface $\Sigma$. If $u_*[\Sigma]$ has positive self-intersection number, then $(X,J)$ is tamed by a symplectic form $\omega$ which is $G$-invariant, $J$-invariant, and Poincaré dual to $u_*[\Sigma]$.
\end{eg}

\begin{proof}
Define
\begin{alignat*}{2}
\ff\!:G \times \Sigma&\lra G, &\qquad
\ff(g,z)&=g,\\
\ev\!:G \times \Sigma&\lra X, &\qquad
\ev(g,z)&=g(u(z)).
\end{alignat*}
Let $\d\mu$ be an invariant volume form on $G$ with $\int_G \d\mu = 1$ (i.e. the Haar measure). As $G$ acts transitively on~$X$, $\ev$ is a submersion. By Lemma \ref{submersion_pushforward_lmm}, $\omega \bydefinition \ev_*\ff^*(\d\mu)$ is a well-defined 2-form on~$X$. Let $\alpha \in \Omega^2(X)$.
By~(\ref{pushforward_acting_on_vector_eqn}) and (\ref{integrated_family_eqn}),
\begin{equation}\label{group_eg_eqn}
T_\omega(\alpha) = \int_{g \in G} \bigg(\int_\Sigma (gu)^*\alpha\bigg) \d\mu.
\end{equation}
By (\ref{group_eg_eqn}) and the $G$-invariance of $\d\mu$, $\omega$ is $G$-invariant. By (\ref{group_eg_eqn}) and the pseudoholomorphicity of~$gu$, $T_\omega(\alpha) =0$ if $\alpha$ is $J$-anti-invariant and $T(\alpha) \geq 0$ if $\alpha$ is semipositive. Therefore, by Lemmas~\ref{invariant_dual_lmm} and $\ref{semipositive_dual_lmm}$, $\omega$ is $J$-invariant and semipositive.
As $G$ is connected, the maps $gu$ are homotopic to $u$. Therefore, if $\alpha$ is closed,
$$\int_{g \in G} \bigg(\int_\Sigma (gu)^*\alpha\bigg) \d\mu = \bigg(\int_\Sigma u^*\alpha \bigg) \bigg(\int_G \d\mu \bigg) = \int_\Sigma u^*\alpha.$$
Thus, $\omega$ is Poincaré dual to $u_*[\Sigma]$. Since  $u_*[\Sigma]$ has positive self-intersection number, $\int_X \omega^2 > 0$, so $\omega$ has full rank at some point of $X$. As $\omega$ is $G$-invariant and $G$ acts transitively on $X$, $\omega$ has full rank everywhere. Therefore, it is a symplectic form taming $J$.
\end{proof}

\subsection{Structure currents}\label{structure_cycles_subsection}

In this section, we recall some basic facts from Sullivan's theory of structure currents. The primary reference is \cite{Sullivan}. Structure currents are used in two distinct ways in this paper. The first is that complex cycles, which form the conceptual core of the present work, are structure currents. The second is that structure currents provide the necessary geometric control on the ``bad'' set appearing in the proof of Proposition~\ref{PD_prp}.\\

A \textit{compact convex cone} in a locally convex topological vector space $V$ is a convex cone $\Lambda$ so that there exists a linear functional $L$ on $V$ such that $L(\Lambda-0) \!\subset\! (0,\infty)$ and $L^{-1}(1)\!\cap\!\Lambda$ is compact. A \textnormal{cone structure} on a manifold $X$ is a closed subset $A \subset X$ and a continuous family of compact convex cones $\Lambda_x \!\subset\!\bigwedge^p(T_x X)$ with $x\!\in\!A$ for some $p \!\in\!\Z^{\geq 0}$.

\begin{dfn}
Let $\Lambda$ be a cone structure on a manifold $X$. A current $T$ on $X$ is a \textit{structure current} with respect to $\Lambda$ or a \textit{$\Lambda$-current} if it belongs to the closed convex cone generated by the Dirac currents associated to multivectors in $\Lambda$. A structure current $T$ is called a \textit{structure cycle} if $\partial T=0$.
\end{dfn}

\begin{rmk}
There is a subtlety in the above definition. The closure is taken with respect to the strong dual topology, not the (perhaps more familiar) weak topology. However, this does not actually present a problem, due to the following result of Schwartz \cite[Chapitre~III, Théorème~XIII]{Schwartz}: a sequence of currents converges in the strong topology if and only if it converges in the weak topology. That is, a sequence of $k$-currents $T_i$ converges to $T$ in the strong topology if $T_i(\alpha)$ converges to $T(\alpha)$ for every $k$-form~$\alpha$. It is therefore often easy to check that a current is a structure current.
\end{rmk}

\begin{rmk}
We refer to a cone structure $\Lambda'$ as being \textit{slightly larger} than another cone structure~$\Lambda$ if $\Lambda\!-\!0$ is contained in the interior of $\Lambda'$, as subsets of $\bigwedge^pTX$. This means, in particular, that the closed subset of $X$ on which $\Lambda'$ is defined is larger than that for $\Lambda$. However, in the context of this paper, the underlying closed subsets are somewhat immaterial and this subtlety does not really make a difference.
\end{rmk}

If $(X,J)$ is a closed almost complex manifold and $\Lambda$ is the cone of complex bivectors, then Sullivan calls $\Lambda$-cycles \textit{complex cycles}. For this specific type of structure cycle, we find it more convenient to use the following alternative definition. The proof of Proposition~\ref{taming_prp} carries through unchanged using Definition~\ref{complex_cycle_dfn} (see \cite[Proposition~1.7]{Cattalani}).

\begin{dfn}\label{complex_cycle_dfn}
A 2-form $\alpha$ on an almost complex manifold is \textit{strictly positive} if $\alpha(v, Jv) > 0$ for all $v \neq 0 \in T_xX$ and $x\in X$. A $2$-current $T$ on a compact almost complex manifold $X$ is a \textit{complex cycle} if $\partial T = 0$ and $T(\alpha) > 0$ for all strictly positive 2-forms.
\end{dfn}

\begin{eg}
A closed pseudoholomorphic curve gives rise to a complex cycle, as in Example~\ref{submanifold_current_eg}. Such complex cycles are exceptionally well-behaved; see \cite{Gue} for more pathological examples.
\end{eg}

Adding an arbitrarily small positive form to a semipositive form results in a positive form, so every semipositive form is a limit of positive forms. Therefore, $T(\alpha) \geq 0$ for every complex cycle $T$ on $X$ and every semipositive form $\alpha$. This fact is used in the proof of Lemma~\ref{pos_on_supp_lmm} below.

\begin{lmm}\label{pos_on_supp_lmm}
Let $\alpha$ be a semipositive form on $X$ and $T$ be a complex cycle on $X$. If $\alpha$ is strictly positive somewhere on $\supp(T)$, then $T(\alpha) > 0$.
\end{lmm}

\begin{proof}
Suppose that $\alpha$ is strictly positive at $x \in \supp(T)$. Let $U \subset X$ be a precompact open neighborhood of $x$ such that $\alpha$
is strictly positive on $\overline{U}$. Choose a form \hbox{$\beta \in \Omega^{2}_c(U)$} with \hbox{$T(\beta) > 0$.} By the compactness
of $\overline{U}$, there exists $\varepsilon >0$ such that $\alpha - \varepsilon \beta$ is a semipositive form on a neighborhood of $\supp(T)$. Therefore,
$$T(\alpha) = \varepsilon T(\beta) + T(\alpha - \varepsilon \beta) > 0,$$
as claimed.
\end{proof}

The rest of this section is devoted to some results about general structure currents. The following results are applied in this paper not to complex cycles, but to other structure cycles. It is well-known that every loop in Euclidean space bounds a 2-chain: one simply extends radially. The following lemma highlights a geometric aspect of this radial extension.

\begin{lmm}\label{cone_filling_lmm}
Let $X$ be a smooth manifold, $\Lambda$ be a cone structure on $X$, and $\iota: \overline{D}\lra X$ be a smooth embedding of a closed disk such that the image of $\d\iota$ lies in $\Lambda$. Then, for any slightly larger cone structure~$\Lambda'$, there exists $\varepsilon > 0$ such that, if $f:S^1 \lra X$ is a smooth map $\varepsilon$-close to $\iota|_{\partial\overline{D}}$ in $C^1$, then there is a $\Lambda'$-current $T$ with $\partial T = f_*[S^1]$.
\end{lmm}

\begin{proof}
By taking appropriate coordinates, we may assume $\iota$ is the inclusion of the set
$$\overline{D} =\{(x_1,x_2,0,\dots,0) \in \R^n \,:\, x_1^2+x_2^2\leq 1 \} \subset \R^n.$$
In particular, $\iota(rx) = r\iota(x)$ for $x \in \overline{D}$ and $r\in[0,1]$. In polar coordinates, define
$$\widehat{f}: \overline{D} \lra \R^n, \quad \widehat{f}(r,\theta) := rf(\theta).$$
The map $\widehat{f}$ is smooth on $\overline{D}-0$. If $\widehat{f}|_{\partial\overline{D}} := f$ is $C^1$-close to $\iota|_{\partial\overline{D}}$, then $\widehat{f}$ is $C^1$-close to $\iota$ on $\overline{D}-0$ by the scaling symmetry of $\widehat{f}$ and $\iota$. If $\widehat{f}$ is sufficiently close to $\iota$ in $C^1$ on $\overline{D}-0$, then the image of $\d\widehat{f}$ lies in $\Lambda'$. For $r \in (0,1)$, let $A_r$ be the annulus resulting from removing a disk of radius $r$ from $\overline{D}$. As $\widehat{f}|_{A_r}$ is smooth, the $\Lambda'$-current
$$T_r: \Omega_c^2(X) \lra \R, \quad T_r(\alpha) := \int_{A_r}\widehat{f}^*\alpha,$$
is well-defined. Let
$T := \lim_{r \rightarrow 0} T_r;$
it is a $\Lambda'$-current. By Stokes' theorem, $\partial T =f_*[S^1]$.
\end{proof}

\begin{rmk}\label{family_of_boundaries_rmk}
As the space of $\Lambda'$-currents is a closed convex cone, it follows from Lemma~\ref{cone_filling_lmm} that any current in the convex hull of the currents of integration over curves $\varepsilon$-close in $C^1$ to $\iota|_{\partial\overline{D}}$ is the boundary of a $\Lambda'$-current. This applies, for instance, to the integral over a measured family thereof, in the sense of Lemma~\ref{integrated_family_lmm}.
\end{rmk}

The next lemma is essentially \cite[Proposition~I.9]{Sullivan}. Although it is not emphasized, it is clear that the smoothing procedure described can be applied just on an open set, leaving the rest unchanged; i.e. smoothing can be done in a relative manner. That is the content of Lemma~\ref{smoothing_lmm}.

\begin{lmm}\label{smoothing_lmm}
Let $X$ be a smooth manifold, $K \subset X$ be compact, and $\Lambda$ be a cone structure on $X$. Let $T$ be a closed current such that $T|_{\textnormal{int}\,K}$ is a $\Lambda$-current and $T$ is smooth on a neighborhood of $X-\textnormal{int}\,K$. For any slightly larger cone structure~$\Lambda'$, there is an exact current $S$ supported in $\textnormal{int}\,K$ such that $T+S$ is a smooth $\Lambda'$-current.
\end{lmm}

The proof of the next lemma follows the same line as that of Lemma~\ref{semipositive_dual_lmm}.

\begin{lmm}\label{cone_positivity_lmm}
Let $X$ be an oriented 4-manifold, $\dvol$ be a volume form on $X$, and $\Lambda$ be a cone structure on $X$. Let $\omega$ be a smooth 2-form such that $\curindby{\omega}$ is a $\Lambda$-current. If $x\!\in\!X$ and $v\!\in\!\bigwedge^2(T_xX)$ is a bivector such that $\dvol(v,w) > 0$ for every $w\!\in\!\Lambda_x$, then $\omega(v) \geq 0$.
\end{lmm}

\begin{proof}
Suppose $v$ satisfies $\omega(v) \!<\! 0$ and $\dvol(v,w) > 0$ for every $w\!\in\!\Lambda_x$. Extend it to a compactly supported bivector field $\eta$ on $X$ such that $\omega(\eta(x))\!\leq\!0$ and $\dvol(\eta(x),w)\!\geq\!0$ for every $x\!\in\!X$ and every~$w \!\in\! \Lambda$.
As $\omega(v) < 0$,
$$\curindby{\omega}\big(\dvol(\eta(x),-)\big)=\int_X \omega \wedge \dvol(\eta(x),-)=\int_X\omega(\eta(x))\dvol< 0,$$
contradicting the fact that $\curindby{\omega}$ is a $\Lambda$-current.
\end{proof}

\section{Preliminaries: curves}\label{deformation_sec}

\noindent
In this section, we recall the necessary perturbation theory for pseudoholomorphic curves. Proposition~\ref{CurveMS_prp} below collects the input on the local structure 
of moduli spaces of $J$-holomorphic curves needed for our purposes.
It is a variation on \cite[Theorem~2]{HLS} and \cite[Lemma~1.5.1]{IvSh2}. 
We deduce it from \cite{HLS,IvSh2,Sh} in Appendix \ref{CurveMS_sec}.\\

Let $(X,J)$ be an almost complex manifold. We recall that a \textit{pseudoholomorphic curve} is a pseudoholomorphic map $u: \Sigma \lra X$ from a Riemann surface, considered up to reparameterization. If $\Sigma$ is connected, the curve is called \textit{irreducible}.
We call a half-dimensional submanifold $Y\!\subset\!X$ \textit{totally real} 
if \hbox{$TY\!\!\cap\!J(TY)\!=\!Y$}.
By the Carleman Similarity Principle \cite[Theorem~2.3.5]{MS}, 
a non-constant $J$-holomorphic map \hbox{$u\!:\Si\!\lra\!X$}
from a connected Riemann surface, possibly with boundary,
determines a complex line subbundle \hbox{$\cT\!u\!\subset\!u^*TX$} containing~$\nd u(T\Si)$. 
We denote~by
$$\cN u\bydefinitiontwo \frac{u^*TX}{\cT u}\lra \Si$$
the \textit{normal bundle} of~$u$.
If $u|_{\prt\Si}$ is an embedding with the image contained in a totally real submanifold $Y\!\subset\!X$,
then the \textit{normal bundle}
$$\cN_Y(\prt u)\bydefinitiontwo\frac{\{u|_{\prt\Si}\}^{\!*}TY}{\nd u(T(\prt\Si)\!)}\lra \prt\Si$$
of~$\prt u\!\bydefinitiontwo\!u|_{\prt\Si}$ in~$Y$ is a real subbundle of~$\cN u|_{\prt\Si}$
(the embedding condition can be bypassed in general).
If, in addition, $\Si$ is compact, we denote by~$\mu_Y^{\cN}(u)$ the Maslov index 
of the bundle pair~$(\cN u,\cN_Y(\prt u)\!)$ over~$(\Si,\prt\Si)$;
see \cite[Appendix~C.3]{MS}.\\

\noindent
We define the \textit{genus} $g(\Si)$ of a compact connected Riemann surface $\Si$ 
with $b$~boundary components~by
$$2-2g(\Si)=\chi(\Si)+b,$$
where $\chi(\Si)$ is the Euler characteristic of $\Si$.
If $(X,J)$ is an almost complex manifold, $Y\!\subset\!X$ is a submanifold,
and $g,b\!\in\!\Z^{\ge0}$, we denote
\begin{enumerate}[label=$\bullet$,leftmargin=*]

\item by~$\fM_{(g,b)}(X,Y)$ the moduli space of equivalence classes of 
non-constant $J$-holomorphic maps \hbox{$u\!:\Si\!\lra\!X$} from compact connected 
genus~$g$ Riemann surfaces with~$b$ boundary components such that $u(\prt\Si)\!\subset\!Y$, 

\item by~$\fM_{(g,b),1}(X,Y)$ the moduli space of equivalence classes of pairs~$(u,z_0)$
with \hbox{$u\!:\Si\!\lra\!X$} as above and $z_0\!\in\!\Si$, and 

\item by $\fM_{(g,b),1}^{\circ}(X,Y)\!\subset\!\fM_{(g,b),1}(X,Y)$ the subspace of pairs
$[u,z_0]$ as above so that $z_0\!\not\in\!\prt\Si$.

\end{enumerate}
Define
\begin{alignat*}{2}
\ff\!:\fM_{(g,b),1}(X,Y)&\lra \fM_{(g,b)}(X,Y), &\qquad
\ff\big([u,z_0]\big)&=[u],\\
\ev\!:\fM_{(g,b),1}(X,Y)&\lra X, &\qquad
\ev\big([u,z_0]\big)&=u(z_0).
\end{alignat*}


\begin{prp}\label{CurveMS_prp}
Suppose $(X,J)$ is an almost complex 4-manifold, 
$Y\!\subset\!X$ is a totally real surface, \hbox{$g,b\!\in\!\Z^{\ge0}$},
and \hbox{$[u\!:\Si\!\lra\!X]\!\in\!\fM_{(g,b)}(X,Y)$}.
If $u|_{\prt\Si}$ is an embedding and \hbox{$\mu_Y^{\cN}(u)\!\ge\!4g\!+\!2b\!+\!1$}, then
\begin{enumerate}[label=(\arabic*),leftmargin=*]

\item\label{CurveMS0_it}
a neighborhood~$\fM_0$ of $[u]$ in~$\fM_{(g,b)}(X,Y)$ is a smooth manifold; 

\item\label{CurveMS1_it} the subspace $\fM_1\!\bydefinitiontwo\!\ff^{-1}(\fM_0)$
of~$\fM_{(g,b),1}(X,Y)$ is a smooth manifold with boundary and the maps
\BE{CurveMS1_e}\ff\!:\fM_1\lra\fM_0 \quad\hbox{and}\quad
\ev\!:\fM_1^{\circ}\!\bydefinitiontwo\!\fM_1\!\cap\!\fM_{(g,b),1}^{\circ}(X,Y)\lra X\!-\!Y\EE
are smooth submersions;

\item\label{CurveMS10_it} for every $x_0\!\in\!X\!-\!Y$,
the subspace $\fM_1^{\circ}(x_0)\!\bydefinitiontwo\!\ev^{-1}(x_0)\!\cap\!\fM_1$ of~$\fM_1^{\circ}$ 
is a smooth submanifold with
\BE{CurveMS10_e}\dim\fM_1^{\circ}(x_0)=\dim\fM_1-2\,;\EE

\item\label{CurveMS11_it} for every $x_0\!\in\!X\!-\!Y$ and $v_0\!\in\!T_{x_0}X$ nonzero,
the subspace $\fM_1^{\circ}(v_0)\!\subset\!\fM_1^{\circ}(x_0)$ of pairs $[u',z_0]$ such that 
$v_0\!\in\!\cT u'|_{z_0}$ is a smooth submanifold with
\BE{CurveMS11_e}\dim\fM_1^{\circ}(v_0)=\dim\fM_1-4\,.\EE

\end{enumerate}
\end{prp}

We note that the assumption that $X$ has dimension $4$ is essential. In higher dimensions, no condition on the Maslov index alone is sufficient to guarantee that the conclusions of Proposition~\ref{CurveMS_prp} hold; see \cite[Remark~(3)]{HLS}. We also emphasize that a submersion is not necessarily surjective.

\begin{dfn}
Let $(X,J)$ be an almost complex manifold. A closed pseudoholomorphic curve $C \subset X$ is \textit{free} if it admits a pseudoholomorphic parameterization $u: \Sigma \lra X$ which satisfies \ref{CurveMS0_it} and \ref{CurveMS1_it} of Proposition \ref{CurveMS_prp} with $Y = \emptyset$.
\end{dfn}

\begin{eg}
The lines in $\mathbb{CP}^2$ and the smooth fibers of a surface bundle over a surface are free.
\end{eg}

If an immersed surface $S$ has $n$ points of self-intersection, all of which are positive transverse double points, then $S$ is said to have $n$ \textit{nodes}. An immersed surface $S$ is said to have $n$ nodes if some arbitrarily small perturbation of $S$ has $n$ nodes. This notion is referred to as having $\delta$-invariant $n$ in some contexts in algebraic geometry.

\begin{eg}\label{free_eg}
Let $(X,J)$ be an almost complex 4-manifold and $C \subset X$ be an immersed pseudoholomorphic curve of genus $g$ with $n$ nodes. Let $u: \Sigma \lra X$ be a pseudoholomorphic parameterization of $C$. It follows from the proof of Proposition \ref{CurveMS_prp} in Appendix~\ref{CurveMS_sec} that $C$ is free~if
$$4g - 1 \leq \mu_\emptyset^\cN(u) = 2\langle c_1(\mathcal{N}u), [\Sigma] \rangle.$$
This and the Adjunction Inequality \cite[Theorem~2.6.4]{MS} imply that if $C$ has self-intersection number at least $2g + 2n$, then it is free.
\end{eg}

Proposition \ref{CurveMS_prp} shows that a pseudoholomorphic curve in an almost complex 4-manifold with sufficiently large Maslov index admits many deformations. Lemma~\ref{boundary_conditions_lmm} below shows that boundary conditions can be chosen to make the Maslov index arbitrarily large.

\begin{lmm}\label{boundary_conditions_lmm}
Let $(X,J)$ be a closed almost complex 4-manifold and $u: \Sigma \lra X$ be a nonconstant pseudoholomorphic map from a closed connected Riemann surface. For any $M \in \mathbb{Z}^{\geq 0}$ and open subset $U \subset X$ such that $u(\Sigma) \cap U \neq \emptyset$, there exists a totally real surface $Y \subset U$ and a curve $\gamma$ bounding an open disk $D \subset \Sigma$ such that $u(\overline{D}) \subset U$, $u(\gamma) \subset Y$, and 
\begin{equation}
\label{boundary_conditions_eqn}
\mu_Y^\cN(u|_{\Sigma - D}) \geq M.
\end{equation}
\end{lmm}

\begin{proof}
Pick a simple closed curve $\gamma \subset u^{-1}(U)$ bounding an open disk $D \subset \Sigma$ so that $u|_{\overline{D}}$ is an embedding.  The normal bundle $\mathcal{N}u$ trivializes over $\overline{D}$. Associate $\mathcal{N}u|_{\overline{D}}$ with a complement to $du(T\overline{D})$ in $TX$. Take a section $\xi$ of $\mathcal{N}u$ along $\gamma$ such that the winding number around the zero section is $2\langle c_1(\mathcal{N}u), [\Sigma]\rangle - M$. Since $u$ is pseudoholomorphic, the subbundle of $TX$ generated by $\xi$ and $T\gamma$ is totally real. This is an open property, so flowing for a short time under the exponential map of $\xi$ yields a totally real surface $Y \subset U$. By \cite[Theorem C.3.5]{MS},
$$2\langle c_1(\mathcal{N}u), [\Sigma]\rangle = \mu_\emptyset^\cN(u) = \mu_Y^\cN(u|_{\Sigma - D}) + \mu_Y^\cN(u|_{\overline{D}}) = \mu_Y^\cN(u|_{\Sigma - D}) + (2\langle c_1(\mathcal{N}u),[\Sigma]\rangle - M).$$
This gives (\ref{boundary_conditions_eqn}).
\end{proof}

\section{Construction of Poincaré dual}\label{PD_sec}

In this section, we use integration along the fiber to produce a well-behaved Poincaré dual to a pseudoholomorphic curve. The main result of this section is Proposition \ref{PD_prp}. Its proof is based on the proof of Lemma \ref{free_curve_PD_lmm}, but with additional complications. If the curve $C$ fails to be sufficiently flexible, we have to vary it as a curve with boundary. Lemma \ref{boundary_conditions_lmm} shows that boundary conditions can be chosen to make the curve as flexible as one wishes. Integration along the fiber then produces a current which is well-behaved away from a small ``bad'' set. We then appeal to Lemma~\ref{bad_set_lmm} to control the resulting current on this small set.

\begin{lmm}\label{free_curve_PD_lmm}
Let $(X,J)$ be a closed almost complex 4-manifold and $C$ be a closed pseudoholomorphic curve on $X$. If $C$ is free, then it is Poincaré dual to a semipositive form $\omega$.
\end{lmm}

\begin{proof}
Let $u_0: \Sigma \lra X$ be a pseudoholomorphic parameterization of $C$. With notation as in Proposition~\ref{CurveMS_prp}, let
$$\mathfrak{f}: \M_1 \lra \M_0~~\hbox{and}~~\textnormal{ev}: \M_1 = \M_1^{\circ} \lra X$$
be the forgetful and evaluation maps, respectively. Since $C$ is free, $\M_0$ and $\M_1^\circ$ are smooth manifolds and $\textnormal{ev}$ is a submersion. Let $\d\nu$ be a compactly supported nonnegative form on $\M_0$ such that $\int_{\M_0} \d\nu = 1$. By Lemma \ref{submersion_pushforward_lmm} with $U = X$, $\omega = \textnormal{ev}_* \mathfrak{f}^* \d\nu$ is a well-defined (smooth) 2-form on $X$.\\

Let $\alpha \!\in\! \Omega^2(X)$. By (\ref{integrated_family_eqn}),
$$\curindby{\omega}(\alpha) = \int_{u \in \M_0} \bigg( \int_{\Sigma} u^*\alpha \bigg) \d\nu.$$
As $u$ is pseudoholomorphic for every $u \!\in\! \M_0$, $\curindby{\omega}(\alpha) \!\geq\! 0$ for every semipositive 2-form $\alpha$. Therefore,~$\omega$ is semipositive by Lemma~\ref{semipositive_dual_lmm}. Since $u$ is homotopic to $u_0$ for all $u \in \M_0$,
$$\int_X\omega \wedge \alpha =\curindby{\omega}(\alpha) =  \int_{u \in \M_0} \bigg( \int_{\Sigma} u^*\alpha \bigg) \d\nu = \bigg( \int_C \alpha \bigg) \bigg(\int_{\M_0}\d\nu \bigg) = \int_\Sigma u_0^* \alpha$$
if $\alpha$ is a closed $2$-form on $X$. Therefore, $\omega$ is Poincaré dual to $C$.
\end{proof}

Let $(X,J)$ be an almost complex 4-manifold with a Riemannian metric and
$$\eta := \xi \wedge J\xi \in \Gamma\big(X,\bigwedge\nolimits^{\!2}TX\big)$$
be a continuous \textit{field of unit complex bivectors}, i.e.~$|\eta|=1$. For $\delta\geq0$, define
\begin{equation}\label{Lambda_delta_eqn}
\conestr{\delta} := \Bigr\{w \in \bigwedge\nolimits^{\!2}T_xX: x \in X, ~~\big|w-|w|\eta(x)\big| \leq\delta|w| \Bigl\} \subset \bigwedge\nolimits^{\!2}TX.
\end{equation}
For any closed subset $A \!\subset\!X$, $\conestr{\delta}|_A$ is a cone structure on $X$.

\begin{lmm}\label{bad_set_lmm}
Let $\varphi$ be a semipositive form on $X$ and $A\!\subset\!X$ be a compact subset so that \hbox{$\varphi(\eta(x))\!>\!0$} for all $x\!\in\!A$. There exists $\delta\!>\!0$ such that for every 2-form $\omega$ so that $\curindby{\omega}$ is a~\hbox{$\conestr{\delta}|_A$-current}, there exists $K\!>\!0$ such that $\omega\!+\!K\varphi$ is a semipositive form on $X$. 
\end{lmm}

\begin{proof}
Let $\dvol$ be the volume form on $X$ determined by the Riemannian metric and $J$. Define
$$\mathbb{G}_2^+(X) := \{ u\wedge Ju: u \in TX,~|u\wedge Ju|=1\}
\subset \bigwedge\nolimits^{\!2}TX.$$
If $\delta\geq0$ and $x \in X$, then
\begin{equation}\label{vol_cone_ineq}
\dvol(v,w) \geq \Big(\dvol\big(\eta(x),v\big)-\delta\Big)|w| \quad \forall v \in \mathbb{G}_2^+(X)|_x,~~w\in\conestr{\delta}|_x\,.
\end{equation}
Since distinct elements of $\mathbb{G}_2^+(X)$ determine disjoint (except at $0$) tangent 2-planes on $X$, the continuous functions
$$\mathbb{G}_2^+(X) \lra \R^{\geq0},\quad v \mapsto \big|v- \eta(x)\big|,~\dvol\big(\eta(x),v \big) \quad \forall v\in \mathbb{G}_2^+(X)|_x,~x\in X,$$
vanish only along $\eta$. Since $\mathbb{G}_2^+(X)|_A$ is compact, it follows that for every $\varepsilon>0$, there exists $\delta>0$ such that
$$x\in A,~~ v \in \mathbb{G}_2^+(X)|_x,~~\dvol\big(\eta(x),v \big)\leq \delta \implies v\in \conestr{\varepsilon}|_x\,.$$
Along with (\ref{vol_cone_ineq}), the last implication gives
\begin{equation}\label{vol_cone_implication_ineg}
x\in A, ~~ v \in \mathbb{G}_2^+(X)|_x,~~w \in \conestr{\delta}|_x,~~\dvol(v,w) \leq0 \implies v \in \conestr{\varepsilon}|_x\,.
\end{equation}
Since $\mathbb{G}_2^+(X)|_A$ is compact and $\varphi>0$ on $\conestr{0}|_A$, there exist $\varepsilon,c>0$ such that $\varphi \geq c$ on $\conestr{\varepsilon} \cap\mathbb{G}_2^+(X)|_A$. Let $\delta>0$ be as in (\ref{vol_cone_implication_ineg}) and assume that $\omega$ is a 2-form on $X$ such that $\curindby{\omega}$ is a $\conestr{\delta}|_A$-current. Let $x\!\in \!A$ and $v \!\in\! \mathbb{G}_2^+(X)$. By (\ref{vol_cone_implication_ineg}), either $\varphi(v)\geq c$ or $\dvol(v,w)>0$ for all $w\!\in\! \conestr{\delta}|_x$.
By Lemma~\ref{cone_positivity_lmm}, $\omega(v)\geq0$ in the latter case. Thus, for each $v \!\in\! \mathbb{G}_2^+(X)|_A$, either $\varphi(v)\geq c$ or $\omega(v) \geq 0$.\\

Since $\mathbb{G}_2^+(X)|_A$ is compact,
$$K := -\inf_{\substack{v \in \mathbb{G}_2^+(X)|_A\\ \omega(v) < 0}} \frac{\omega(v)}{\varphi(v)}$$
is a finite positive number (if $\omega(v)\geq0$ for all $v\in \mathbb{G}_2^+(X)|_A$, simply take $K=0$).
By the definition of $K$ and the semipositivity of $\varphi$, $\omega+K\varphi$ is nonnegative on $\mathbb{G}_2^+(X)|_A$. Since the support of $\omega$ is contained in $A$ (because it is a $\conestr{\delta}|_A$-current), it follows that $\omega+K\varphi$ is a semipositive form on~$X$.
\end{proof}

\begin{figure}
\centering
\tikzset{every picture/.style={line width=0.75pt}} 

\begin{tikzpicture}[x=0.75pt,y=0.75pt,yscale=-1,xscale=1]
\begin{scope}
\clip(433.19,28.73) rectangle (633.19,228.73);
\draw  [dash pattern={on 4.5pt off 4.5pt}] (458.08,128.73) .. controls (458.08,89.67) and (491.71,58) .. (533.19,58) .. controls (574.67,58) and (608.3,89.67) .. (608.3,128.73) .. controls (608.3,167.8) and (574.67,199.47) .. (533.19,199.47) .. controls (491.71,199.47) and (458.08,167.8) .. (458.08,128.73) -- cycle ;
\draw    (568.11,133.66) .. controls (585.67,133.66) and (646.65,163.85) .. (718.31,223.14) .. controls (789.97,282.43) and (828.19,368.55) .. (826.47,441) .. controls (824.75,513.45) and (806.61,590.85) .. (742.35,653.2) .. controls (678.09,715.55) and (535.5,749.4) .. (517.02,749.4) .. controls (498.55,749.4) and (398,735.05) .. (321.74,675.84) .. controls (245.47,616.63) and (206.42,516.63) .. (204.57,441) .. controls (202.71,365.37) and (256.65,260.36) .. (327.75,206.16) .. controls (398.84,151.97) and (478.74,146.75) .. (498.93,146.75) ;
\draw    (568.11,120.56) .. controls (688.29,129.05) and (652.66,152.53) .. (724.32,211.82) .. controls (795.98,271.11) and (766.38,375.92) .. (832.48,429.68) .. controls (898.58,483.44) and (811.45,647.54) .. (748.36,641.88) .. controls (685.27,636.23) and (574.1,817.3) .. (523.03,738.08) .. controls (471.96,658.86) and (404.01,723.73) .. (327.75,664.52) .. controls (251.48,605.31) and (261.65,457.98) .. (210.58,429.68) .. controls (159.5,401.39) and (249.63,200.5) .. (333.76,194.85) .. controls (417.88,189.19) and (351.78,138.26) .. (499,132.6) ;
\draw    (568.11,133.66) .. controls (667.26,142.14) and (625.11,235.51) .. (715.24,215.71) .. controls (805.37,195.9) and (844.43,357.18) .. (790.35,365.66) .. controls (736.27,374.15) and (916.54,575.04) .. (823.4,569.38) .. controls (730.27,563.72) and (739.28,790.07) .. (601.08,730.65) .. controls (462.88,671.23) and (406.95,772.88) .. (330.68,713.67) .. controls (254.42,654.46) and (221.38,642.18) .. (219.52,566.55) .. controls (217.67,490.92) and (246.56,360) .. (231.54,272.29) .. controls (216.52,184.58) and (369.74,136.49) .. (498.93,133.66) ;
\draw    (568.11,107.47) .. controls (643.22,110.3) and (658.67,135.56) .. (730.33,194.85) .. controls (801.99,254.13) and (840.21,311.97) .. (838.49,384.41) .. controls (836.77,456.86) and (848.67,593.68) .. (784.41,656.03) .. controls (720.15,718.38) and (502.45,689.98) .. (483.97,689.98) .. controls (465.5,689.98) and (333.76,740.91) .. (297.7,712.62) .. controls (261.65,684.32) and (222.59,463.63) .. (177.53,418.37) .. controls (132.46,373.1) and (333.76,84.5) .. (381.83,158.06) .. controls (429.9,231.63) and (366.74,104.64) .. (498.93,107.47) ;
\draw    (568.11,146.75) .. controls (670.26,158.06) and (586.06,117.74) .. (670.18,148.86) .. controls (754.3,179.98) and (677.91,169.78) .. (676.19,242.23) .. controls (674.47,314.67) and (825.93,232.45) .. (853.44,290.33) .. controls (880.96,348.21) and (815.48,534.3) .. (775.33,573.26) .. controls (735.18,612.22) and (755.71,701.98) .. (703.23,731.71) .. controls (650.75,761.43) and (534.98,601.56) .. (528.97,714.73) .. controls (522.96,827.9) and (409.95,700.38) .. (333.69,641.17) .. controls (257.42,581.96) and (218.37,481.96) .. (216.52,406.33) .. controls (214.67,330.7) and (268.61,225.69) .. (339.7,171.49) .. controls (410.79,117.3) and (478.74,120.56) .. (498.93,120.56) ;
\end{scope}
\begin{scope}
\draw    (533.19,128.73) -- (498.93,107.47) ;
\draw    (533.19,128.73) -- (498.93,120.56) ;
\draw    (533.19,128.73) -- (498.93,133.66) ;
\draw    (533.19,128.73) -- (499,132.6) ;
\draw    (533.19,128.73) -- (498.93,146.75) ;
\draw    (568.11,146.75) -- (533.19,128.73) ;
\draw    (533.19,128.73) -- (568.11,107.47) ;
\draw    (568.11,120.56) -- (533.19,128.73) ;
\draw    (568.11,133.66) -- (533.19,128.73) ;
\draw    (568.11,133.66) -- (533.19,128.73) ;
\draw  [dash pattern={on 4.5pt off 4.5pt}] (299.17,48.14) .. controls (299.17,38.12) and (308.44,30) .. (319.88,30) .. controls (331.32,30) and (340.59,38.12) .. (340.59,48.14) .. controls (340.59,58.15) and (331.32,66.27) .. (319.88,66.27) .. controls (308.44,66.27) and (299.17,58.15) .. (299.17,48.14) -- cycle ;
\draw    (329.51,49.4) .. controls (334.35,49.4) and (351.17,57.14) .. (370.93,72.34) .. controls (390.69,87.55) and (401.23,109.63) .. (400.75,128.21) .. controls (400.28,146.78) and (395.28,166.63) .. (377.56,182.62) .. controls (359.84,198.6) and (320.51,207.28) .. (315.42,207.28) .. controls (310.32,207.28) and (282.6,203.6) .. (261.57,188.42) .. controls (240.54,173.24) and (229.77,147.6) .. (229.26,128.21) .. controls (228.75,108.81) and (243.62,81.89) .. (263.22,67.99) .. controls (282.83,54.09) and (304.86,52.76) .. (310.43,52.76) ;
\draw    (329.51,46.04) .. controls (362.65,48.22) and (352.82,54.24) .. (372.59,69.44) .. controls (392.35,84.64) and (384.18,111.52) .. (402.41,125.3) .. controls (420.64,139.09) and (396.61,181.16) .. (379.21,179.71) .. controls (361.81,178.26) and (331.16,224.69) .. (317.08,204.38) .. controls (302.99,184.07) and (284.26,200.7) .. (263.22,185.52) .. controls (242.19,170.34) and (245,132.56) .. (230.91,125.3) .. controls (216.83,118.05) and (241.68,66.54) .. (264.88,65.09) .. controls (288.08,63.64) and (269.85,50.58) .. (310.45,49.13) ;
\draw    (329.51,49.4) .. controls (356.85,51.58) and (345.23,75.52) .. (370.08,70.44) .. controls (394.94,65.36) and (405.71,106.71) .. (390.79,108.89) .. controls (375.88,111.06) and (425.59,162.57) .. (399.91,161.12) .. controls (374.22,159.67) and (376.71,217.71) .. (338.6,202.47) .. controls (300.49,187.24) and (285.07,213.3) .. (264.04,198.12) .. controls (243,182.94) and (233.89,179.79) .. (233.38,160.4) .. controls (232.87,141) and (240.84,107.44) .. (236.69,84.95) .. controls (232.55,62.46) and (274.81,50.12) .. (310.43,49.4) ;
\draw    (329.51,42.69) .. controls (350.22,43.41) and (354.48,49.89) .. (374.24,65.09) .. controls (394,80.29) and (404.54,95.12) .. (404.07,113.7) .. controls (403.59,132.27) and (406.88,167.35) .. (389.16,183.34) .. controls (371.43,199.33) and (311.4,192.05) .. (306.31,192.05) .. controls (301.21,192.05) and (264.88,205.11) .. (254.94,197.85) .. controls (245,190.6) and (234.23,134.01) .. (221.8,122.4) .. controls (209.37,110.79) and (264.88,36.8) .. (278.14,55.66) .. controls (291.39,74.52) and (273.98,41.96) .. (310.43,42.69) ;
\draw    (329.51,52.76) .. controls (357.68,55.66) and (334.46,45.32) .. (357.65,53.3) .. controls (380.85,61.28) and (359.79,58.66) .. (359.31,77.24) .. controls (358.84,95.81) and (400.6,74.73) .. (408.19,89.57) .. controls (415.78,104.41) and (397.72,152.13) .. (386.65,162.12) .. controls (375.58,172.11) and (381.24,195.12) .. (366.77,202.75) .. controls (352.3,210.37) and (320.37,169.37) .. (318.72,198.39) .. controls (317.06,227.41) and (285.89,194.71) .. (264.86,179.53) .. controls (243.83,164.35) and (233.06,138.71) .. (232.55,119.32) .. controls (232.04,99.92) and (246.92,73) .. (266.52,59.1) .. controls (286.13,45.2) and (304.86,46.04) .. (310.43,46.04) ;
\draw   (20,125) .. controls (20,83.58) and (53.58,50) .. (95,50) .. controls (136.42,50) and (170,83.58) .. (170,125) .. controls (170,166.42) and (136.42,200) .. (95,200) .. controls (53.58,200) and (20,166.42) .. (20,125) -- cycle ;
\draw  [dash pattern={on 4.5pt off 4.5pt}] (80,50) .. controls (80,38.95) and (88.95,30) .. (100,30) .. controls (111.05,30) and (120,38.95) .. (120,50) .. controls (120,61.05) and (111.05,70) .. (100,70) .. controls (88.95,70) and (80,61.05) .. (80,50) -- cycle ;
\end{scope}

\draw (118,19) node [anchor=north west][inner sep=0.75pt]   [align=left] {$U$};
\draw (345,21) node [anchor=north west][inner sep=0.75pt]   [align=left] {$U$};
\draw (583,48) node [anchor=north west][inner sep=0.75pt]   [align=left] {$U$};
\draw (134,193) node [anchor=north west][inner sep=0.75pt]   [align=left] {$C$};
\draw (385,196) node [anchor=north west][inner sep=0.75pt]   [align=left] {$T_\varepsilon$};
\draw (608,162) node [anchor=north west][inner sep=0.75pt]   [align=left] {$T_\varepsilon $};
\draw (529,158) node [anchor=north west][inner sep=0.75pt]   [align=left] {$T$};

\end{tikzpicture}
\caption{The construction of the form $\omega_C'$ in the proof of Proposition~\ref{PD_prp}. We start with a curve~$C$ and an open set $U$ intersecting $C$. We remove a small disk from $C$ and vary it in a family of pseudoholomorphic curves with boundary. Integration over this family produces the current $T_\varepsilon$, the boundary of which is supported in $U$. We then fill in the boundary by a current $T$. Smoothing the current~$T_\varepsilon+T$ results in the form $\omega_C'$.}
\end{figure}

\begin{prp}\label{PD_prp}
Let $(X,J)$ be a closed almost complex 4-manifold. Let $C$ be a closed irreducible pseudoholomorphic curve in $X$ and $\varphi$ a semipositive form on $X$ such that
$$\int_C \varphi > 0.$$
Then, there exist a closed 2-form $\omega_C$ on $X$ and a constant $K > 0$ such that
\begin{enumerate}[label=(\arabic*),leftmargin=*]
    \item\label{main_prp_PD} $\omega_C$ is Poincaré dual to $C$,
    \item\label{main_prp_J} $\omega_C$ is $J$-invariant on $X -\supp(\varphi)$, 
    \item\label{main_prp_semi} $\omega_C + K\varphi$ is a semipositive form, and
    \item\label{main_prp_posi} $\omega_C + K\varphi$ is strictly positive on a neighborhood of $C$.
    
\end{enumerate}
\end{prp}

\begin{proof}
We show that there exist a closed 2-form $\omega_C'$ and a constant $K' > 0$ such that $\omega_C' + K'\varphi$ is strictly positive on a neighborhood of $C - \supp(\varphi)$ and \ref{main_prp_PD}, \ref{main_prp_J}, and \ref{main_prp_semi} hold with $\omega_C$ and $K$ replaced by $\omega_C'$ and $K'$, respectively. In the last paragraph, we use this to prove the full claim.\\

Fix a Riemannian metric on $X$. Let $u_0: \Sigma \lra X$ be a pseudoholomorphic parameterization of $C$. As $C$ is irreducible, $\Sigma$ is connected. By \cite[Lemma 2.4.1, Proposition 2.5.1]{MS}, $u_0$ is an embedding away from a finite set of points. By assumption, there is an embedded point of $C$ at which $u_0^*\varphi > 0$. Take a small contractible neighborhood $U \subset X$ of this point and a (continuous) field $\eta :=\xi\!\wedge\! J \xi$ of unit complex bivectors on $\overline{U}$ tangent to $C$ along $C$ so that $\varphi(\eta(x)) \!>\! 0$ for all $x \in \overline{U}$. With $\conestr{\delta} \!\subset\! \bigwedge^{\!2}TX$ as in (\ref{Lambda_delta_eqn}), fix $\delta$ as in Lemma~\ref{bad_set_lmm} with $A=\overline{U}$.\\

Let $g$ be the genus of $\Sigma$ and $M = 4g + 3$. Let $D \!\subset \!\Sigma$ and $Y\!\subset\!U$ be as in Lemma~\ref{boundary_conditions_lmm} with $u=u_0$. With the notation as in Proposition~\ref{CurveMS_prp} with $u=u_0|_{\Si-D}$, let
$$\mathfrak{f}: \M_1 \lra \M_0~~\hbox{and}~~\textnormal{ev}: \M_1 \lra X$$
be the forgetful and evaluation maps, respectively; both maps are smooth, while $\f$ is also a proper submersion. For each $\varepsilon\!>\!0$, let $\d\volumeformm_\varepsilon$ be a nonnegative form of top degree on $\M_0$ that is supported on curves $\varepsilon$-close in $C^1$ to $u_0|_{\Sigma-D}$, is nonzero at $u_0|_{\Sigma-D}$, and satisfies $\int_{\M_0}\d\volumeformm_\varepsilon =1$.\\

Define the $2$-current $T_\varepsilon \bydefinition \textnormal{ev}_*( \mathfrak{f}^* \d\volumeformm_\varepsilon)$ on $X$ as in (\ref{pushforward_eqn}) with $f=\ev$.
By (\ref{integrated_family_eqn}),
\begin{equation}
\label{average_integral_eqn}
T_\varepsilon(\alpha) = \int_{u \in\M_0}\bigg(\int_{\Sigma - D} u^*\alpha\bigg)\d\volumeformm_\varepsilon \quad \forall \alpha \in \Omega^2(X).
\end{equation}

We now show that $T_\varepsilon$ is well-behaved on $X-U$. Since $\ev(\partial\M_1) \!\subset \!Y\!\subset\!U$, $T_\varepsilon$ is smooth on a neighborhood $W_\varepsilon
\!\subset\!X$ of $X-U$ by Lemma~\ref{submersion_pushforward_lmm}, i.e.~$T_\varepsilon|_{W_\varepsilon}=T_{\omega_\varepsilon}$ for some (smooth) 2-form $\omega_\varepsilon$ on $W_\varepsilon$. Since each map $u$ is $J$-holomorphic, (\ref{average_integral_eqn}) implies that $T_\varepsilon(\alpha) =0$ for $J$-anti-invariant~$\alpha$. By Lemma~\ref{invariant_dual_lmm}, $\omega_\varepsilon$ is thus a $J$-invariant 2-form on $W_\varepsilon$.\\

Let $x_0 \in  W_\varepsilon$ and $v_0 \in T_{x_0}X$. By (\ref{pushforward_acting_on_vector_eqn}),
$$\omega_\varepsilon(v_0,Jv_0) = \int_{\textnormal{ev}^{-1}(x_0)}\{\mathfrak{f}^*\d\volumeformm_\varepsilon\}(-,\widetilde{v_0},\widetilde{Jv_0}),$$
where $\wt{v_0}, \wt{Jv_0} \in \Gamma(\ev^{-1}(x_0);T\M_1)$ are the lifts of $v_0$ and $Jv_0$, respectively, along $\d\ev$.
The sign of $\omega_\varepsilon(v_0,Jv_0)$ depends on the orientation of the fiber, which we now describe. The manifold $\M_0$ is oriented so $\d\volumeformm_\varepsilon$ is nonnegative and $\M_1$ is oriented by giving the fibers of the forgetful map $\mathfrak{f}$ their complex orientations as Riemann surfaces. The manifold $X$ is oriented by the almost complex structure $J$. The orientation of the fiber of ev is determined by requiring that the wedge product of the lift of an orientation form on $X$ and an orientation form on $\textnormal{ev}^{-1}(x_0)$ gives a form along $\textnormal{ev}^{-1}(x_0)$ which agrees with the orientation on $\M_1$. It follows that $\{\mathfrak{f}^*\d\volumeformm_\varepsilon\}(-,\widetilde{v_0},\widetilde{Jv_0})|_{T\ev^{-1}(x_0)}$ at a point $(u, z) \in \textnormal{ev}^{-1}(x_0)$ is zero if $v_0 \in \mathcal{T} u|_z$ and a positively oriented form in the top degree otherwise. Therefore, $\omega_\varepsilon(v_0, Jv_0) \geq 0$, i.e.~$\omega_\varepsilon$ is semipositive on $W_\varepsilon$.\\

Suppose, in addition, that $x_0 \in u_0(\Si-\overline{D})$. By (\ref{CurveMS10_e}) and (\ref{CurveMS11_e}), there then exists a surface $u \in \M_0$ arbitrarily close to $u_0|_{\Si-D}$ passing through $x_0$ transverse to $v_0\wedge Jv_0$. Therefore, $\omega_\varepsilon$ is a strictly positive form along $C - U$. As strict positivity is an open condition, $\omega_\varepsilon$ is strictly positive on a neighborhood $W'_\varepsilon \subset W_\varepsilon$ of $C-U$.\\

We have shown that $\omega_\varepsilon$ is well-behaved on $X-U$. We now modify it on $U$ to produce a 2-form~$\omega_C'$ which is well-behaved everywhere. By (\ref{average_integral_eqn}) and Stokes' theorem,
$$\{\partial T_\varepsilon \}(\beta) := T_\varepsilon(\d\beta) = -\int_{u\in\M_0}\bigg(\int_{\partial\overline{D}}u^*\beta\bigg)\d\volumeformm_\varepsilon \quad \forall \beta\in \Omega^1(X).$$
By Lemma~\ref{cone_filling_lmm} with $\Lambda = \conestr{\delta}$ and Remark~\ref{family_of_boundaries_rmk}, there exists a current~$T$ supported in $U$ with $\partial T = - \partial T_\varepsilon$. For $\delta_1 \in (0,\delta)$, if $\varepsilon$ is sufficiently small, both $T_\varepsilon$ and $T$ are $\conestr{\delta_1}$-currents in $\overline{U}$. Thus, $T_\varepsilon + T$ is a cycle which is a $\conestr{\delta_1}$-current on $\overline{U}$ and smooth on a neighborhood of $X - U$. Therefore, by Lemma~\ref{smoothing_lmm}, there exists an exact current $S$ supported in $U$ such that
$$T_\varepsilon' :=T_\varepsilon +T+S$$ is a smooth, closed current that is a $\conestr{\delta}$-current on $\overline{U}$. Let $\omega_C'$ be the 2-form on $X$ so that $T_\varepsilon'=\curindby{\omega_C'}$. By Lemma~\ref{bad_set_lmm}, there exists $K'>0$ such that $\omega_C' + K'\varphi$ is semipositive on $\overline{U}$. As $\omega_C'$ is equal to $\omega_\varepsilon$ on a neighborhood of $X-U$, $\omega_C'$ is $J$-invariant on $X-U$ and $\omega'_C + K'\varphi$ is semipositive everywhere and strictly positive on a neighborhood of $C-U$.\\

The last property of $\omega'_C$ that remains to be proved is that it is Poincaré dual to $C$. As the curves in $\M_0$ are homologous relative to $U$, it follows from (\ref{average_integral_eqn}) that
$$T_\varepsilon(\alpha) = \int_{u \in\M_0}\bigg(\int_{\Sigma - D} u^*\alpha\bigg)\d\volumeformm_\varepsilon = \bigg(\int_{\Sigma-D}u_1^*\alpha\bigg)\bigg(\int_{\M_0}\volumeformm_
\varepsilon\bigg) = \int_{\Sigma-D}u_1^*\alpha$$
for $\alpha\in \Omega^2(X)$ supported in $X-U$. This implies the second equality in
$$[\omega'_C] = [T_\varepsilon] = u_{0*}[\Sigma-D] = [C] \in H_2(X,U;\R).$$
As $U$ is contractible, the long exact sequence in relative homology implies that the equality $[\omega'_C] = [C]$ in $H_2(X,U;\R)$ lifts to an equality $[\omega'_C] = [C]$ in $H_2(X;\R)$. Thus, $\omega'_C$ is Poincaré dual to~$C$.\\

We have thus constructed a 2-form $\omega_C'$ on $X$ and $K'>0$ such that $\omega_C' + K'\varphi$ is positive on a neighborhood of $C - \supp(\varphi)$ and \ref{main_prp_PD}, \ref{main_prp_J}, and \ref{main_prp_semi} hold with $\omega$ replaced by $\omega_C'$ and $K$ by $K'$. We can decompose $\varphi$ as $\varphi = \varphi_1 + \varphi_2 + \varphi_3$, where $\varphi_i$ is semipositive for $i \in \{1,2,3\}$, $\int_C \varphi_i > 0$ for $i \in \{1,2\}$, and $\supp(\varphi_1) \cap \supp(\varphi_2) = \emptyset$. Construct $\omega_1'$ and $K_1' > 0$ with respect to $\varphi_1$ and $\omega_2'$ and $K_2' > 0$ with respect to $\varphi_2$, as before. The form $\omega_C := \frac{1}{2}(\omega_1' + \omega_2')$ and the positive constant $K := \frac{1}{2}(K_1' + K_2')$ satisfy \ref{main_prp_PD}, \ref{main_prp_J}, \ref{main_prp_semi}, and \ref{main_prp_posi}, as claimed.
\end{proof}

\section{Proofs of the main theorems}\label{main_proofs_sec}

In this section, we prove the results stated in Section \ref{Introduction_sec}. Proposition \ref{PD_prp} implies Theorems~\ref{stable_pos_int_thm} and~\ref{PosInter_thm}. Theorem \ref{PosInter_thm} then implies Theorem \ref{Curves2Symp_thm}.

\begin{proof}[{\bf{\emph{Proof of Theorem \ref{stable_pos_int_thm}}}}]
As $C$ is compact, there is a finite set of balls $B_i \subset X$ of diameter $\varepsilon$ intersecting $C$ such that $C \subset \bigcup B_i$. For each $i$, take a semipositive form $\varphi_i$ supported inside $B_i$ such that $\int_C \varphi_i > 0$. By Proposition \ref{PD_prp}, for each $i$, there exist a closed 2-form $\omega_i$, a positive constant~$K_i$, and an open set $U_i \subset X$ containing $C$ such that $\omega_i$ is Poincaré dual to $C$ and $\omega_i + K_i \varphi_i$ is a semipositive form which is strictly positive on $U_i$. The intersection $\bigcap U_i$ is an open set containing~$C$. As $C$ is compact, there exists $\delta > 0$ such that every point of $X$ $\delta$-close to $C$ is contained in $\bigcap U_i$. Let $S$ be a closed pseudoholomorphic curve such that $\langle [S], [C] \rangle \leq 0$ and some point of $S$ is $\delta$-close to $C$. Since $S$ intersects each $U_i$,
$$0 < \int_S (\omega_i + K_i \varphi_i) = \langle [S], [C] \rangle + K_i \int_S \varphi_i \leq K_i \int_S \varphi_i.$$
Therefore, $S$ intersects the support of each $\varphi_i$ and thus each $B_i$. As the $B_i$ cover $C$, each point of $C$ is $\varepsilon$-close to $S$.
\end{proof}

\begin{proof}[{\bf{\emph{Proof of Corollary~\ref{Hausdorff_crl}}}}]
Suppose $\langle[S],[C] \rangle > 0$ and for every $\delta>0$ there exists a closed pseudoholomorphic curve $S'\subset X$ which is $\delta$-close to $S$ in the Hausdorff metric and satisfies $\langle[S],[C] \rangle \leq 0$. Given $\varepsilon > 0$, let $\delta > 0$ be as in Theorem~\ref{stable_pos_int_thm} and $S'\subset X$ be a closed pseudoholomorphic curve which is $\min(\varepsilon,\delta)$-close to $S$ in the Hausdorff metric and satisfies $\langle[S'], [C] \rangle \leq 0$. Since $S\cap C\neq \emptyset$, as least one point of $S'$ is $\delta$-close to $C$ and thus every point of $C$ is $\varepsilon$-close to $S'$ and $2\varepsilon$-close to $S$. As $\varepsilon$ is arbitrary, $C \subset S$.
\end{proof}

\begin{proof}[{\bf{\emph{Proof of Theorem \ref{PosInter_thm}}}}]
By Proposition \ref{PD_prp}, there exist a closed 2-form $\omega_C$ Poincaré dual to $C$ and a positive constant $K$ such that $\omega_C + K \varphi$ is a semipositive form which is strictly positive on $C$ and therefore somewhere on $\supp(T)$. By Lemma \ref{pos_on_supp_lmm},
$$0 < T(\omega_C + K\varphi) = T(\omega_C) + K T(\varphi) = T(\omega_C) = \langle [T] , [C] \rangle,$$
as claimed.
\end{proof}

\begin{proof}[{\bf{\emph{Proof of Proposition \ref{semipositive_cone_prp}}}}]
Let $T$ be a nonzero complex cycle on $X$. There exists a closed connected pseudoholomorphic curve $C$ such that $C \cap \supp(T) \neq \emptyset$ and $\int_C \varphi > 0$.
As $C$ is closed and connected, there are irreducible components $C_0, \dots, C_k$ of $C$ such that
\begin{equation}\label{chain_of_curves_eqn}
C_0 \cap \supp(T) \neq \emptyset, \quad \int_{C_k} \varphi > 0, \quad \textnormal{and} \quad  C_{i-1} \cap C_i \neq \emptyset \quad \forall i = 1, \dots, k. 
\end{equation}
Let $i \in \{0, 1, \dots, k \}$ be the largest integer such that $C_i \cap \supp(T) \neq \emptyset$. If $i < k$, then $C_i \not\subset \supp(T)$ by the last condition in (\ref{chain_of_curves_eqn}) and thus $\langle [T], [C_i] \rangle \neq 0$ by Corollary \ref{pos_int_crl}. If $i = k$, then either $T(\varphi) \neq 0$ or $\langle [T], [C_k] \rangle \neq 0$, by Theorem \ref{PosInter_thm}. Therefore, $[T] \neq 0 \in H_2(X,\R)$ and $(X,J)$ admits a taming symplectic structure by Proposition \ref{taming_prp}.
\end{proof}

\begin{proof}[{\bf{\emph{Proof of Theorem \ref{Curves2Symp_thm}}}}]
We first show that $\ref{umerc_it} \Rightarrow \ref{erc_it} \Rightarrow \ref{imgn_it} \Rightarrow \ref{fc_it}$.
Suppose $X$ has uncountably many embedded rational curves. As $X$ has only countably many integer homology classes, at least $2$ of these curves lie in the same homology class, which thus has nonnegative self-intersection number by the positivity of intersections. Therefore, $\ref{umerc_it} \Rightarrow \ref{erc_it}$. Trivially, $\ref{erc_it} \Rightarrow \ref{imgn_it}$. It follows from Example \ref{free_eg} that $\ref{imgn_it} \Rightarrow \ref{fc_it}$. Therefore, it suffices to consider the case when $X$ contains a free curve, $C_{\textnormal{free}}$.
Let $\varphi$ be a semipositive Poincaré dual to $C_\textnormal{free}$ as in Lemma \ref{free_curve_PD_lmm}. For every $x \in X$, there exists a closed connected pseudoholomorphic curve $C_x$ passing through $x$ and intersecting $C_\textnormal{free}$ positively. Therefore, $\int_{C_x} \varphi > 0$ and the result follows from Proposition \ref{semipositive_cone_prp}. 
%
\end{proof}


\appendix

\section{Proof of Proposition~\ref{CurveMS_prp}}
\label{CurveMS_sec}

\noindent
With the assumptions as in Proposition~\ref{CurveMS_prp}, let
$$q_u\!:u^*TX\lra\cN u \quad\hbox{and}\quad
S_u\bydefinitiontwo \big\{z\!\in\!\Si\!:\nd_zu\!=\!0\big\}\subset \Si\!-\!\prt\Si$$
be the quotient projection and the subset of the non-immersive points of~$u$, respectively.
For each $z\!\in\!S_u$, we define $\ord_zu\!\in\!\Z^+$ to be the order of~$z$
as a zero of the section~$\nd u$ of the complex line bundle $T^*\Si\!\otimes_{\C}\!\cT u$
over~$\Si$; this is well-defined by the Carleman Similarity Principle \cite[Theorem~2.3.5]{MS}.
Define
\begin{equation*}\begin{split}
\Ga(T\Si)&=\big\{\ze\!\in\!\Ga(\Si;T\Si)\!:
\ze|_{\prt\Si}\!\in\!\Ga\big(\prt\Si;T(\prt\Si)\!\big)\!\big\}, \\
\Ga(\cT u;\prt\Si)&=\big\{\xi\!\in\!\Ga(\Si;\cT u)\!:
\xi|_{\prt\Si}\!\in\!\Ga\big(\prt\Si;\nd u(T(\prt\Si)\!)\!\big)\!\big\},
\qquad\hbox{and}\\
\Ga(\cN u;Y)&=\big\{\xi\!\in\!\Ga(\Si;\cN u)\!:
\xi|_{\prt\Si}\!\in\!\Ga\big(\prt\Si;\cN_Y(\prt u)\!\big)\!\big\}.
\end{split}\end{equation*}

\vspace{.15in}

We denote by
\begin{equation*}\begin{split}
&D_{Y;J;u}\!:\Ga(u;Y)\!\bydefinitiontwo\!\big\{\xi\!\in\!\Ga(\Si;u^*TX)\!:
\xi|_{\prt\Si}\!\in\!\Ga\big(\prt\Si;\{\prt u\}^*TY\big)\!\big\}\\
&\hspace{2.5in}\lra \Ga^{0,1}_J(u)\!=\!\Ga\big(\Si;(T^*\Si)^{0,1}\!\otimes_{\C}\!u^*TX\big)
\end{split}\end{equation*}
the linearization of the $\dbar_J$-operator on the space of smooth maps 
from~$(\Si,\prt\Si)$ to~$(X,Y)$; see \cite[Section~1.2]{IvSh2}, for example. 
This is a real Cauchy-Riemann operator such~that 
\begin{equation*}\begin{split}
D_{Y;J;u}\big(\nd u\big(\Ga(T\Si)\!\big)\!\big)&\subset
\nd u\big(\Ga(\Si;(T^*\Si)^{0,1}\!\otimes_{\C}\!T\Si)\!\big)
\qquad\hbox{and}\\
D_{Y;J;u}\big(\Ga(\cT u;\prt\Si)\!\big)&\subset
\Ga\big(\Si;(T^*\Si)^{0,1}\!\otimes_{\C}\!\cT u\big);
\end{split}\end{equation*}
see \cite[Section~1.3]{IvSh2}.
Thus, $D_{Y;J;u}$ induces linear operators
\begin{equation*}\begin{split}
\ov{D}_{Y;J;u}\!:\Ga(u;Y)\big/\nd u\big(\Ga(T\Si)\!\big)
&\lra \Ga^{0,1}_J(u)\big/\nd u\big(\Ga(\Si;(T^*\Si)^{0,1}\!\otimes_{\C}\!T\Si)\!\big)\\
\hbox{and}\qquad
D_{Y;J;u}^{\cN}\!:\Ga(\cN u;Y)&\lra \Ga(\Si;(T^*\Si)^{0,1}\!\otimes_{\C}\!\cN u\big).
\end{split}\end{equation*}
The projection $q_u$ induces homomorphisms 
\BE{quker_e}  \wt{q}_{u;0}\!:\ker \ov{D}_{Y;J;u}\lra \ker D_{Y;J;u}^{\cN}
\quad\hbox{and}\quad
\wt{q}_{u;1}\!:\cok\,\ov{D}_{Y;J;u}\lra\cok\,D_{Y;J;u}^{\cN}\,.\EE

\vspace{.15in}

\noindent 
For $z_0\!\in\!\Si$, let 
\begin{alignat*}{2}
\Ga_{-z_0}(T\Si)&=\big\{\ze\!\in\!\Ga(T\Si)\!:\xi(z_0))\!=\!0\big\}, &\quad
\Ga_{-z_0}(\cN u;Y)&=\big\{\xi\!\in\!\Ga(\cN u;Y)\!:\xi(z_0)\!=\!0\big\},\\
\Ga_{-z_0}(u;Y)&=\big\{\xi\!\in\!\Ga(u;Y)\!:\xi(z_0)\!=\!0\big\},
 &\quad
\Ga_{-2z_0}(\cN u;Y)&=\big\{\xi\!\in\!\Ga_{-z_0}(\cN u;Y)\!:\na\xi|_{z_0}\!=\!0\big\}.
\end{alignat*}
We denote by
\begin{equation*}\begin{split}
\wt{D}_{Y;J;u;-z_0}\!:\Ga(u;Y)\big/\nd u\big(\Ga_{-z_0}(T\Si)\!\big)
&\lra \Ga^{0,1}_J(u)\big/\nd u\big(\Ga(\Si;(T^*\Si)^{0,1}\!\otimes_{\C}\!T\Si)\!\big)
\quad\hbox{and}\\
\ov{D}_{Y;J;u;-z_0}\!:\Ga_{-z_0}(u;Y)\big/\nd u\big(\Ga_{-z_0}(T\Si)\!\big)
&\lra \Ga^{0,1}_J(u)\big/\nd u\big(\Ga(\Si;(T^*\Si)^{0,1}\!\otimes_{\C}\!T\Si)\!\big)
\end{split}\end{equation*}
the linear operators induced by~$D_{Y;J;u}$.\\

\noindent
Suppose the operator~$\ov{D}_{Y;J;u}$ is surjective.
By the Implicit Function Theorem for Banach manifolds, as in \cite[Section~3.5]{MS},
a neighborhood~$\fM_0$ of~$[u]$ in~$\fM_{(g,b)}(X,Y)$ is then a smooth manifold.
Furthermore, the subspace $\fM_1\!\bydefinitiontwo\!\ff^{-1}(\fM_0)$ of~$\fM_{(g,b),1}(X,Y)$ 
is a smooth manifold with boundary so~that 
\BE{ndev_e1a}T_{[u',z_0]}\fM_1= \ker\wt{D}_{Y;J;u';-z_0} 
\quad\forall\,[u',z_0]\!\in\!\fM_1\EE
and the first map in~\eref{CurveMS1_e}
is a fiber bundle so that the fiber over $[u']\!\in\!\fM_0$ is the domain of~$u'$.
The second map in~\eref{CurveMS1_e} is smooth and its differential is given by
\BE{ndev_e1b} \nd_{[u',z_0]}\ev\big([\xi]\big)=\xi(z_0)
\quad\forall~[\xi]\!\in\!\ker\wt{D}_{Y;J;u';-z_0}.\EE

\vspace{.15in}

\noindent
If the operator~$\ov{D}_{Y;J;u}$ is surjective and
the second map in~\eref{CurveMS1_e} is a submersion, then the subspace 
\hbox{$\fM_1^{\circ}(x_0)\!\bydefinitiontwo\!\ev^{-1}(x_0)\!\cap\!\fM_1$} of~$\fM_1^{\circ}$ 
is a smooth submanifold satisfying~\eref{CurveMS10_e} and 
\BE{CurveMS10_e2a} 
T_{[u',z_0]}\!\big(\fM_1^{\circ}(x_0)\!\big)= \ker\ov{D}_{Y;J;u';-z_0} 
\quad\forall\,[u',z_0]\!\in\!\fM_1^{\circ}(x_0).\EE
The differential $\nd_{z_0}u'$ induces a smooth map~$\fd$ from~$\fM_1^{\circ}(x_0)$ to 
$T_{z_0}^*\Si\!\otimes_{\R}\!T_{x_0}X$ such~that
\BE{CurveMS10_e2b} 
q_u\!\circ\!\big\{\!\nd_{[u',z_0]}\fd\big([\xi]\big)\!\!\big\}
=q_u\!\circ\!\na\xi|_{z_0}\quad\forall~[\xi]\!\in\!\ker\ov{D}_{Y;J;u';-z_0}.\EE

\vspace{.15in}

\noindent
By \cite[Lemma~1.5.1]{IvSh2}, the second homomorphism in~\eref{quker_e} is an isomorphism,
the first homomorphism is surjective, and so is the homomorphism
\BE{IbSh2_e3} 
\ker\!\big(\wt{q}_{u;0}\big)\lra \bigoplus_{z\in S_u}\!\!\cT u|_z, \qquad
\xi\lra\big(\xi(z)\!\big)_{z\in S_u}\,.\EE
The statement of \cite[Lemma~1.5.1]{IvSh2} is made for closed surfaces, but 
the proof applies to $J$-holomorphic maps from bordered surfaces as in the statement 
of Proposition~\ref{CurveMS_prp}.
Alternatively, the statement of \cite[Lemma~1.5.1]{IvSh2} can be applied directly
after doubling~$\Si$, $T\Si$, $u^*TX$, and~$\cN u$ along~$\prt\Si$ 
as in the proofs of \cite[Theorem~$2'$]{HLS} and \cite[Section~3]{XCapsSetup}.
By \cite[Theorem~$2'$(ii)]{HLS}, the operator~$D_{Y;J;u}^{\cN}$ is onto if
\hbox{$\mu_Y^{\cN}(u)\!\ge\!4g\!+\!2b\!-\!3$}.
This establishes Proposition~\ref{CurveMS_prp}\ref{CurveMS0_it},
as well as the parts of~\ref{CurveMS1_it} concerning 
the one-marked space~$\fM_1$ and the forgetful morphism~$\ff$.\\

\noindent
We can assume $\mu_Y^{\cN}(u')\!\ge\!\mu_Y^{\cN}(u)\!\ge\!4g\!+\!2b\!-\!3$ 
for all $[u']\!\in\!\fM_0$
(i.e.~the maps~$u'$ close to~$u$ are no more singular than~$u$).
Let $[u',z_0]\!\in\!\fM_1^{\circ}$ and $\Si'$ be the domain of~$u'$.
Since 
$$\big[\nd u'(\ze)\big]\in \ker\wt{D}_{Y;J;u';-z_0} \quad\hbox{and}\quad
\nd_{[u',z_0]}\ev\big(\big[\nd u'(\ze)\big]\big)=\nd_{z_0}u'\big(\ze(z_0)\!\big)
\qquad\forall~\ze\!\in\!\Ga(\Si'),$$
the image of~\eref{ndev_e1b} contains~$\cT u'|_{z_0}$ if $z_0\!\not\in\!S_{u'}$.
If $z_0\!\in\!S_{u'}$, then this is the case by the surjectivity of~\eref{IbSh2_e3}.
By the surjectivity of the first homomorphism in~\eref{quker_e},
$\nd_{[u',z_0]}\ev$ is thus surjective if the homomorphism
$$\ker D_{Y;J;u'}^{\cN}\lra \cN u'|_{z_0}, \qquad \xi\lra\xi(z_0),$$
is surjective.
This is the case if the restriction
$$D_{Y;J;u';-z_0}^{\cN}\!:
\Ga_{-z_0}(\cN u';Y)\lra \Ga(\Si;(T^*\Si')^{0,1}\!\otimes_{\C}\!\cN u'\big)$$
of~$D_{Y;J;u'}^{\cN}$ is surjective.
By the twisting construction of \cite[Lemma~2.4.1]{Sh},
the kernel and cokernel of this operator are isomorphic to
the kernel and cokernel of a real Cauchy-Riemann operator~$D$ as above
with~$\cN u'$ replaced by the complex line bundle $\cN u'\!\otimes_{\C}\!\cO(-z_0)$,
with the same boundary condition on the domain.
The Maslov index of the resulting bundle pair is~$\mu_Y^{\cN}(u')\!-\!2$.
By \cite[Theorem~$2'$(ii)]{HLS}, such an operator~$D$ is thus onto if
\hbox{$\mu_Y^{\cN}(u')\!\ge\!4g\!+\!2b\!-\!1$}.
This establishes the part of Proposition~\ref{CurveMS_prp}\ref{CurveMS1_it} 
concerning the evaluation morphism~$\ev$, as well as~\ref{CurveMS10_it}.\\

\noindent
We now assume that $\mu_Y^{\cN}(u')\!\ge\!\mu_Y^{\cN}(u)\!\ge\!4g\!+\!2b\!-\!1$ 
for all $[u']\!\in\!\fM_0$.
Let $[u',z_0]\!\in\!\fM_1^{\circ}(v_0)$ and $\Si'$ be the domain of~$u'$.
By~\eref{CurveMS10_e2a}, \eref{CurveMS10_e2b}, and
the surjectivity of the first homomorphism in~\eref{quker_e},
the homomorphism 
$$q_{u'}\!\circ\!\nd_{[u',z_0]}\fd\!:
T_{[u',z_0]}\!\big(\fM_1^{\circ}(x_0)\!\big)\lra T_{z_0}^*\Si'\!\otimes_{\C}\!\cN u'|_{z_0}$$
is surjective if the homomorphism
$$\ker D_{Y;J;u';-z_0}^{\cN}\lra T_{z_0}^*\Si'\!\otimes_{\C}\cN u'|_{z_0}, 
\qquad \xi\lra\na\xi|_{z_0},$$
is surjective.
This is the case if the restriction
$$D_{Y;J;u';-2z_0}^{\cN}\!:
\Ga_{-2z_0}(\cN u';Y)\lra \Ga(\Si;(T^*\Si')^{0,1}\!\otimes_{\C}\!\cN u'\big)$$
of~$D_{Y;J;u'}^{\cN}$ is surjective.
By the twisting construction of \cite[Lemma~2.4.1]{Sh},
the kernel and cokernel of this operator are isomorphic to
the kernel and cokernel of a real Cauchy-Riemann operator~$D$ as above
with~$\cN u'$ replaced by the complex line bundle $\cN u'\!\otimes_{\C}\!\cO(-2z_0)$,
with the same boundary condition on the domain.
The Maslov index of the resulting bundle pair is~$\mu_Y^{\cN}(u')\!-\!4$.
By \cite[Theorem~$2'$(ii)]{HLS}, such an operator~$D$ is thus onto if
\hbox{$\mu_Y^{\cN}(u')\!\ge\!4g\!+\!2b\!+\!1$}.
This establishes Proposition~\ref{CurveMS_prp}\ref{CurveMS11_it}.

\vspace{.3in}

{\it Department of Mathematics, Stony Brook University, Stony Brook, NY 11794\\
spencer.cattalani@stonybrook.edu}

\end{document}